\documentclass[12pt,reqno]{amsart}
\usepackage{amssymb}
\usepackage{amsmath,amscd}
\usepackage{enumitem}
\usepackage{subfig}
\usepackage{graphicx}
\usepackage{color}
\usepackage{float}
\usepackage[pagewise]{lineno}
\restylefloat{figure}
\newlength{\dhatheight}

\newcommand\restr[2]{{
  \left.\kern-\nulldelimiterspace 
  #1 
  \littletaller 
  \right|_{#2} 
  }}
\newcommand{\littletaller}{\mathchoice{\vphantom{\big|}}{}{}{}}

\usepackage[usenames,dvipsnames,svgnames,table]{xcolor}
\usepackage[colorlinks=true,
            linkcolor=blue,
            urlcolor=blue,
            citecolor=blue]{hyperref}
\DeclareFontFamily{U}{mathb}{\hyphenchar\font45}
\DeclareFontShape{U}{mathb}{m}{n}{
      <5> <6> <7> <8> <9> <10> gen * mathb
      <10.95> mathb10 <12> <14.4> <17.28> <20.74> <24.88> mathb12
      }{}
\DeclareSymbolFont{mathb}{U}{mathb}{m}{n}
\DeclareFontSubstitution{U}{mathb}{m}{n}

\DeclareMathSymbol{\curvearrowleft}    {\mathrel}{mathb}{"F0}
\DeclareMathSymbol{\curvearrowbotright}{\mathrel}{mathb}{"F4}

\makeatletter
\newcommand\curvearrowleftbotright{%
  \mathord{\mathpalette\@curvearrowleftbotright\relax}%
}
\newcommand\@curvearrowleftbotright[2]{%
  \raisebox{.3ex}{$#1\m@th\curvearrowleft$}%
  \llap{\raisebox{-.3ex}{$#1\m@th\curvearrowbotright$}}%
}
\makeatother

\newtheorem{theorem}{Theorem}[section]
\newtheorem{lemma}[theorem]{Lemma}

\newtheorem{problem}[theorem]{Problem}
\newtheorem{corollary}[theorem]{Corollary}

\theoremstyle{definition}
\newtheorem{definition}[theorem]{Definition}
\theoremstyle{definition}

\theoremstyle{remark}
\newtheorem{remark}[theorem]{Remark}

\numberwithin{equation}{section}
\newcommand{\reg}{{\rm reg}}
\newcommand{\ind}{{\rm ind}}
\newcommand{\Per}{{\rm Per}}

\newcommand{\id}{\operatorname{id}}

\newcommand{\Fix}{{\rm Fix}}

\usepackage[numbers,sort&compress]{natbib}

\begin{document}
\vspace{0.5in}
\title[Fixed point indices of iterates]%
{Fixed point indices of iterates of orientation-reversing homeomorphisms} 

\author{Grzegorz Graff}
\address{Faculty of Applied Physics and Mathematics, Gda\'nsk University of
Technology, Narutowicza 11/12, 80-233 Gda\'nsk, Poland}

\email{grzegorz.graff@pg.edu.pl}

\author{Patryk Topór}
\address{Doctoral School, Gda\'nsk University of
Technology, Narutowicza 11/12, 80-233 Gda\'nsk, Poland, Orcid: 0009-0009-9723-7816}
\email{patryk.topor@pg.edu.pl}

\subjclass[2010]{Primary 37B30, 55M20; Secondary  37E30.}
\keywords{Fixed point index,  Conley index, topological dynamics, homemomorphism}

\thanks{Research supported by the National Science Centre, Poland,
within the grant Sheng 1 UMO-2018/30/Q/ST1/00228.}

\maketitle
\begin{abstract}
We show that any sequence of integers satisfying necessary Dold’s congruences is realized as
the sequence of fixed point indices of the iterates of an orientation-reversing homeomorphism
 of $\mathbb{R}^{m}$ for $m\geq 3$.
 As an element of the construction of the above homeomorphism, we consider the class of boundary-preserving homeomorphisms of $\mathbb{R}^{m}_{+}$ and give the answer to [Problem 10.2, Topol. Methods Nonlinear Anal. 50 (2017), 643–667] providing a complete description of the forms of fixed point indices for this class of maps.
\end{abstract}

\section{Introduction} 
Let $f$ be a self-homeomorphism of $\mathbb{R}^{m}$ for $m\geq 3$ and $p\in \mathbb{R}^{m}$ be a fixed point. We assume that $p$ is isolated  for each iteration of $f$. In such a case, the fixed point index of a map $f^{n}$ at a point $p$, $\ind(f^{n},p)$, is a well-defined integer for each $n$ (we refer the reader to \cite{Marzan} for the formal definition and properties of the fixed point index). This topological device is often used to study local dynamics of a map near fixed points.   To this aim, the knowledge about the behavior of a whole sequence $( \ind(f^{n},p))_{n}$ is crucial (cf. \cite{CMPY, Llibre, Calvez, GrJez, GrJDE, BaBo, GYZhang, GrTopAppl-simpl, GrBasel} for the applications of the sequence of indices to study different classes of maps). 

However, determination of the forms of the sequence $(\ind(f^{n},p))_{n}$ for a given class is usually difficult.  
Since the work of Albrecht Dold \cite{Dold} it is known that every such sequence has to satisfy some congruences, called Dold's relations (see Definition \ref{DoldIntegers} and Theorem \ref{DoldRel}).

In this paper, we consider local indices of homeomorphisms of $\mathbb{R}^{m}$. The case of $m=1$ is trivial, so the first dimension of interest is $m=2$.
 The pioneering work for the class of orientation-preserving homeomorphisms of $\mathbb{R}^{2}$ is the paper of Morton Brown \cite{Brown}. He showed that the fixed point index of the iterated map can only take two different values, one of which is $1$. Later in \cite{LeRoux, PoSa4} the  result of Brown was  developed and all possible forms of $(\ind(f^{n},p))_{n}$  were accurately described for planar homeomorphisms. 
 Let us mention that the strong restrictions for indices of planar homeomorphisms were found in case $\lbrace p\rbrace$ is isolated as an invariant set (see Definition \ref{inv}) \cite{CalvezYoccoz, PoSa2}.
 
 The  problem was investigated also in higher dimensions, i.e. for $m\geq 3$ by Babenko and Bogatyi \cite{BaBo}. It was proved that for a given sequence of integers $(d_{n})_{n}$, which satisfies Dold's congruences, there exists an orientation-preserving self-homeomorphism of $\mathbb{R}^{m}$ such that the equality $\ind(f^{n},p)=d_{n}$ holds for each $n$. However, in their construction, the fixed point $p$ must be a limit of periodic points. 
 This restriction was released in \cite{PoSa} and \cite{PoSa3} in which the authors assumed that $\Fix(f)=\Per(f)=\lbrace p\rbrace$ and showed that there are no further bounds on $(\ind(f^{n},p))_{n}$  other than Dold's relations (cf. also \cite {Corbato_acc} for the continuous case).

The class of orientation-reversing homeomorphisms is not so well-examined. The first work in this area was the paper \cite{Bonino} of Bonino, who showed that for planar orientation-reversing self-homeomorphisms, the fixed point index can take only three values from the set $\lbrace -1,0,1\rbrace$. Later in \cite{GrPrz2} based on the result of Bonino and Dold's congruences, some relationships among elements of the sequence of indices were found. The complete description of the form of indices in two-dimensional case is given in \cite{PoSa4}.

In a more recent work \cite{Corbato} of Hernández-Corbato, Le Calvez and Ruiz del Portal, analysis of  the $3$-dimensional case has been performed 
under the additional assumption that $\lbrace p\rbrace$ is isolated as an invariant set. In particular, strong constraints on the form of the sequences of indices has been found in this case (cf. Theorem \ref{revers}).

The natural question is whether there are any restrictions in the orientation-reversing case in $\mathbb R^3$ (and in higher dimensions) in general case, when the additional assumption on $\lbrace p\rbrace$ used in \cite{Corbato} is dropped.  

The indices of  orientation-reversing homeomorphism represented very regular patterns in the studied cases,  it may suggest that there are also some bounds in dimension $3$ for an arbitrary map in this class. 

In this paper we give the answer to this problem, showing that there are no restrictions, except for Dold's relations for indices of an orientation reversing homeomorphism in $\mathbb R^m$ for $m\geq 3$ (Theorem \ref{thm3}).

Note that indices of orientation-reversing homeomorphisms naturally arise in the context of differential equations, for example they are a useful tool to study periodic solutions via Poincar\'e map (cf. \cite{Po5}).

The second main result of the paper is the following.
In the proof of Theorem \ref{thm3} i.e. in our construction of a homeomorphism with arbitrary indices we consider boundary-preserving homeomorphisms 
 of $\mathbb{R}^{m}_{+}=\lbrace (x,t)\in \mathbb{R}^{m-1}\times [0,\infty)\rbrace$.
  For that class of maps, in Section \ref{sec3},  we  determine all possible forms of pairs $(\ind(f^{n},p),\ind(\bar{f}^{n},p))_{n}$, where $p\in \partial \mathbb{R}_{+}^{m}$ and $\bar{f}=\restr{f}{\partial \mathbb{R}^{m}_{+}}$ (Theorem \ref{thm1} and Corollary \ref{Cor-boundR^m}). In this way, we give a complete solution to Problem 10.2 in \cite{BaWo}. 

\section{Properties of the fixed point indices of iterates \label{sec2}}
Since $\mathbb{R}^{m}$ is homogeneous, we can assume without loss of generality that $0\in \mathbb{R}^{m}$ is an isolated fixed point for each iteration of $f\colon \mathbb{R}^{m}\to \mathbb{R}^{m}$.
\begin{definition}\label{DoldIntegers}
For each natural number $n$, we define integers: \begin{align} i_{n}(f,0)=\sum_{k|n}\mu\left(\frac{n}{k}\right)\ind(f^{k},0),
\end{align}
where $\mu\colon \mathbb{N}\to \lbrace -1,0,1\rbrace$ denotes the classical Möbius function defined by the three following properties: $\mu(1)=1$, $\mu(n)=(-1)^{r}$ if $n$ is a product of $r$ distinct primes and $\mu(n)=0$ otherwise.
\end{definition}
In 1983 A. Dold  proved that integers $i_{n}(f,0)$ are divisible by $n$ for any sequence of fixed point indices of iterates \cite{Dold}:
\begin{theorem}\label{DoldRel} The relation $i_{n}(f,0)\equiv 0\ (\bmod \ n)$ is satisfied for each $n$.
\end{theorem}
In general, for any sequence of integers $(d_{n}) _{n}$, the following congruence: \begin{align}
    \sum_{k|n}\mu\left(\frac{n}{k}\right)d_{k} \equiv 0 \ (\hbox{mod} \ n)
\end{align} is called the Dold's relation. Using the properties of these relations, one may represent sequences of fixed point indices of iterates in the simple form of $k-$periodic expansion using  basic periodic regular sequences described in Definition \ref{reg} below.
\begin{definition}\label{reg}For a given integer $k>0$, we define the regular sequence: \begin{align}\label{regseq}
        \reg_{k}(n)=\begin{cases}
k \ \ \hbox{if} \ \ k  |n, \\
0 \ \ \hbox{otherwise}.
\end{cases}\end{align}
\end{definition}
\begin{theorem} {\normalfont\cite{Marzan}}.
Every arithmetic function $\psi\colon \mathbb{N}\to \mathbb{C}$  can be written uniquely in the periodic expansion form: \begin{align}
        \psi(n)=\sum_{k=1}^{\infty}a_{k}\reg_{k}(n),
    \end{align}
    where $a_{n}=\frac{1}{n}i_{n}(\psi)=\frac{1}{n}\sum_{d|n}\mu\left(\frac{n}{d}\right)\psi(d)$. Moreover, $\psi$ is integral valued and satisfied Dold's relations if and only if $a_{n}(\psi)\in \mathbb{Z}$ or equivalently if $i_{n}(\psi)\equiv 0 \ (\bmod \ n)$.
\end{theorem}
\begin{corollary}
    Any sequence of fixed point indices of iterates can be represented uniquely in the form of periodic expansion: \begin{align} \label{Realization}\ind(f^{n},0)=\sum_{k=1}^{\infty}a_{k}\reg_{k}(n),
\end{align}
where $a_{n}=\frac{1}{n}i_{n}(f,0)\in \mathbb{Z}$. 
\end{corollary}
The elements of the sequence $(a_{n})_{n}$ are called {\it Dold's coefficients}. 
The arithmetical properties of such sequences were discussed in \cite{ByGrWa}.
Let us note that the problem of determining all possible forms of fixed point indices of iterates $( \ind(f^{n},0))_{n}$ can be reformulated into the problem of realizing arbitrary sequence of integers (i.e. Dold's coefficients) by the equation (\ref{Realization}).
\section{Indices of boundary-preserving homeomorphisms \label{sec3}}

Let us define $\mathbb{R}^{m}_{+}=\lbrace (x,t)\in \mathbb{R}^{m-1}\times [0,\infty)\rbrace$ and 
${\partial \mathbb{R}^{m}_{+}}= \lbrace (x,t)\in \mathbb{R}^{m-1}\times \lbrace 0 \rbrace \rbrace$.
In this section, we consider a boundary-preserving homeomorphism $f \colon \mathbb{R}^{m}_{+} \to \mathbb{R}^{m}_{+}$ and its restriction to a boundary, i.e., 
a map $\bar{f}\colon {\partial \mathbb{R}^{m}_{+}} \to {\partial \mathbb{R}^{m}_{+}}$.
 We give the complete description of all possible forms of local fixed point indices of iterates of such pairs of homeomorphisms, i.e. $(\ind(f^n,0), \ind(\bar{f}^n,0))_{n}$.
Our result provides the solution to Problem 10.2 in \cite{BaWo} (Problem \ref{BargeWojcik} below).
\begin{problem}{\normalfont(cf.~~\cite{BaWo})}. Let $f$ be a boundary-preserving homeomorphism of $\mathbb{R}^{m}_{+}$. Determine all possible forms of pairs $(\ind(f^n,0), \ind(\bar{f}^n,0))_{n}$. \label{BargeWojcik}
\end{problem}

It turns out that the only restrictions are in dimension $2$ and $3$ (see Corollary \ref{Cor-boundR^m}). As the restrictions in $\mathbb{R}^{2}_{+}$ are quite obvious and follow from Theorem \ref{forms} given below, 
we will concentrate on the case of
boundary-preserving homeomorphisms of $\mathbb{R}^{3}_{+}$.
Our main result of this section is the following:
\begin{theorem} \label{thm1}
Let $(f,\bar{f})\colon(\mathbb R^{3}_+,\partial \mathbb R^{3}_+)\to (\mathbb
R^{3}_+,\partial \mathbb R_+^{3}$) be a homeomorphism with $0
\in \partial \mathbb{R}^{3}_{+}$ an isolated fixed point for each iteration.
Then the sequences of indices have the following forms:
    \begin{align} \label{31}
        \ind(f^{n},0)=\sum_{k=1}^{\infty} a_{k}\reg_{k}(n), \ \ \ \ \ind(\bar{f}^{n},0)=\reg_{1}(n)+\bar{a}_{d}\reg_{d}(n),
    \end{align}
 where $a_k$ are arbitrary integers, $d>0$, $\bar{a}_{d}$ is an arbitrary integer.
\end{theorem}

The necessary restrictions on the forms of $(\ind(\bar{f}^{n},0))_n$ in Theorem \ref{thm1} result from the following statement.
\begin{theorem}\label{forms}{\normalfont(\cite{LeRoux}, cf.~ also \cite{PoSa4})}. Let $h$ be an orientation-preserving homeomorphism of the plane.~ Then, a sequence of fixed point indices $(\ind(h^{n},0))_{n}$ has the following form:
\begin{align}\ind(h^{n},0)=\reg_{1}(n)+\bar{a}_{d}\reg_{d}(n)
\end{align} for a fixed integer $d>0$ and $\bar{a}_{d}$ an arbitrary integer.
\end{theorem}

In the following part of this section we show that for any integer sequences $(a_{n})_{n=1}^{\infty}$ and  integer $\bar{a}_{n}$ of Dold's coefficients, there is a boundary-preserving homeomorphism $f$ of $\mathbb{R}^{3}_{+}$ such that $(\ind(f^n,0), \ind(\bar{f}^n,0))_n$ is given by the formula (\ref{31}). The subsections \ref{canonical}-\ref{constr} are devoted to the construction of the desired map and in subsection \ref{proofThm1} we give the proof of the Theorem \ref{thm1}, i.e. we verify that the indices of the constructed map satisfy (\ref{31}).
\subsection{Canonical planar flows}\label{canonical}
Let us define a canonical planar flow $H_{p}\colon \mathbb{R}^{2}\times \mathbb{R}\to \mathbb{R}^{2}$ as the flow whose phase portrait around the point $0$ consists of: $2(1-p)$ hyperbolic sectors for $p<1$, $2(p-1)$ elliptic sectors for $p>1$ and a source at point $0$ for $p=1$ (Figure \ref{canflow}).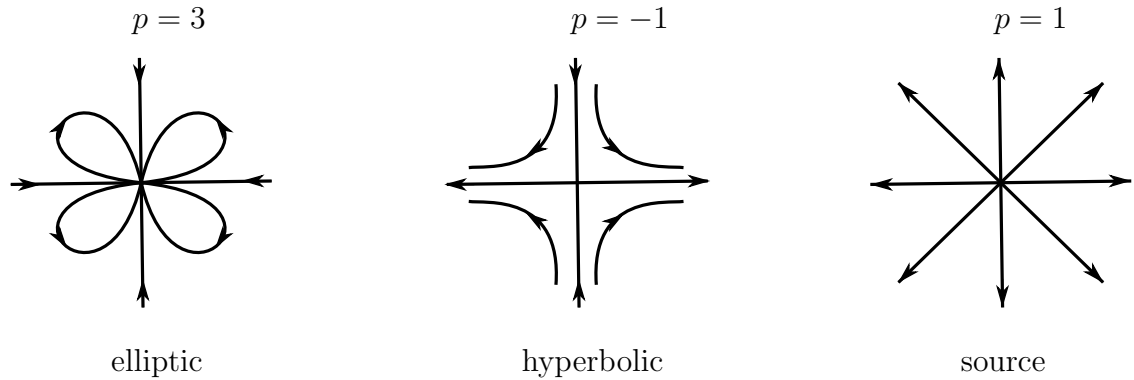
\begin{figure}[H]
\def\svgwidth{1\columnwidth}
\centering
\begingroup%
  \makeatletter%
  \providecommand\color[2][]{%
    \errmessage{(Inkscape) Color is used for the text in Inkscape, but the package 'color.sty' is not loaded}%
    \renewcommand\color[2][]{}%
  }%
  \providecommand\transparent[1]{%
    \errmessage{(Inkscape) Transparency is used (non-zero) for the text in Inkscape, but the package 'transparent.sty' is not loaded}%
    \renewcommand\transparent[1]{}%
  }%
  \providecommand\rotatebox[2]{#2}%
  \newcommand*\fsize{\dimexpr\f@size pt\relax}%
  \newcommand*\lineheight[1]{\fontsize{\fsize}{#1\fsize}\selectfont}%
  \ifx\svgwidth\undefined%
    \setlength{\unitlength}{651.96850394bp}%
    \ifx\svgscale\undefined%
      \relax%
    \else%
      \setlength{\unitlength}{\unitlength * \real{\svgscale}}%
    \fi%
  \else%
    \setlength{\unitlength}{\svgwidth}%
  \fi%
  \global\let\svgwidth\undefined%
  \global\let\svgscale\undefined%
  \makeatother%
  \begin{picture}(1,0.34782609)%
    \lineheight{1}%
    \setlength\tabcolsep{0pt}%
    \put(0.11517616,0.31368034){\makebox(0,0)[lt]{\lineheight{1.25}\smash{\begin{tabular}[t]{l}$k=3$\end{tabular}}}}%
    \put(0.49767231,0.31368034){\makebox(0,0)[lt]{\lineheight{1.25}\smash{\begin{tabular}[t]{l}$k=-1$\end{tabular}}}}%
    \put(0.86695018,0.31368034){\makebox(0,0)[lt]{\lineheight{1.25}\smash{\begin{tabular}[t]{l}$k=1$\end{tabular}}}}%
    \put(0.09658796,0.02922354){\makebox(0,0)[lt]{\lineheight{1.25}\smash{\begin{tabular}[t]{l}$\text{elliptic}$\end{tabular}}}}%
    \put(0.45473159,0.02922354){\makebox(0,0)[lt]{\lineheight{1.25}\smash{\begin{tabular}[t]{l}$\text{hyperbolic}$\end{tabular}}}}%
    \put(0.83772438,0.02922354){\makebox(0,0)[lt]{\lineheight{1.25}\smash{\begin{tabular}[t]{l}$\text{source}$\end{tabular}}}}%
    \put(0,0){\includegraphics[width=\unitlength,page=1]{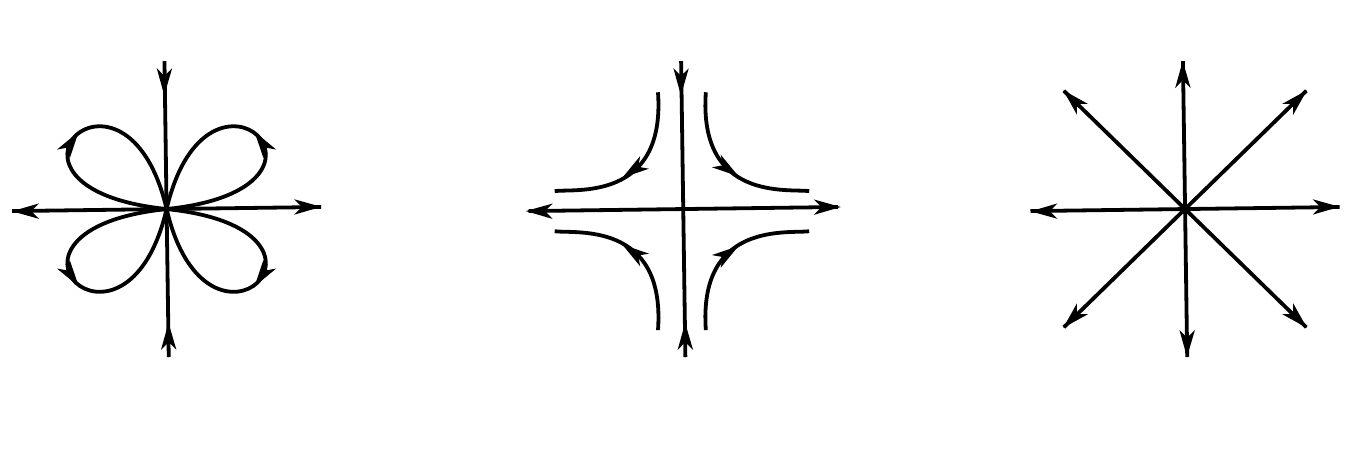}}%
  \end{picture}%
\endgroup%

\caption{Examples of various types of phase portraits of $H_{p}$.}
\label{canflow}
\end{figure}
Note that for any integer $p$ such flows can be described in the standard polar coordinates $(r,\varphi)$ by the following system of equations: 
    \begin{align}\label{system}
\begin{cases}\dot{r}=r\cos{(p-1)\varphi},\\
\dot{\varphi}=\sin{(p-1)\varphi}.\end{cases}
\end{align}
Now, we consider $h_{p}\colon \mathbb{R}^{2}\to \mathbb{R}^{2}$, a time-one map of the flow $H_{p}$. Note that by using the classical Poincar\'e-Bendixson index formula (see \cite{Bendixson}) $\ind(h^{n},0)=1+\frac{e-h}{2}=p$ for each $n$, where $h$ and $e$ denote numbers of hyperbolic and elliptic sectors around $0$, respectively. For a given integer $d>0$, let us define the planar rotation map $O$ by an angle $\frac{2\pi}{d}$ around the origin. Let $a_{d}$ be an arbitrary integer. Observe that for any integer $p$ of the form $1+a_{d}d$, the maps $O$ and $h_{p}$ commute. Therefore: \begin{align}
    \ind((O\circ h_{p})^{n},0)=\ind(O^{n}\circ h_{p}^{n},0)=\reg_{1}(n)+a_{d}\reg_{d}(n).
\end{align}
\subsection{Conical surfaces}\label{sectors}
Let us use the standard spherical coordinates in $\mathbb{R}^{3}_{+}$ given by: \begin{align} \begin{split}
    x&=r\cos{\phi}\cos{\theta},\\
    y&=r\sin{\phi}\cos{\theta},\\
    z&=r\sin{\theta},\end{split}
\end{align}
where $r\in [0,\infty)$ denotes the radial distance and $\varphi \in [0,2\pi]$, $\theta\in [0,\frac{\pi}{2}]$ represent the longitude and the latitude, respectively. We define the main and the intermediate conical surfaces as follows: \begin{align}\begin{split}\mathcal{C}_{k}&=\left\lbrace (r,\varphi,\theta_{k})\mathbin{:} \theta_{k}=\frac{\pi}{4}\left(1-\frac{1}{k}\right)\right\rbrace \ (\hbox{the $k$-th main conical surface}), \\ \mathcal{I}_{k}&=\left\lbrace (r,\varphi,\widetilde{\theta}_{k}) \mathbin{:} \widetilde{\theta}_{k}=\frac{\theta_{k+1}+\theta_{k}}{2} \right\rbrace \ (\hbox{the $k$-th intermediate conical surface}).\end{split}
\end{align} Note that these surfaces converges to the surface $\mathcal{B}=\lbrace (r,\varphi,\theta)\mathbin{:} \theta=\frac{\pi}{4}\rbrace$ (Figure \ref{Cone-Reg}).
\begin{figure}[!h]
\def\svgwidth{1\columnwidth}
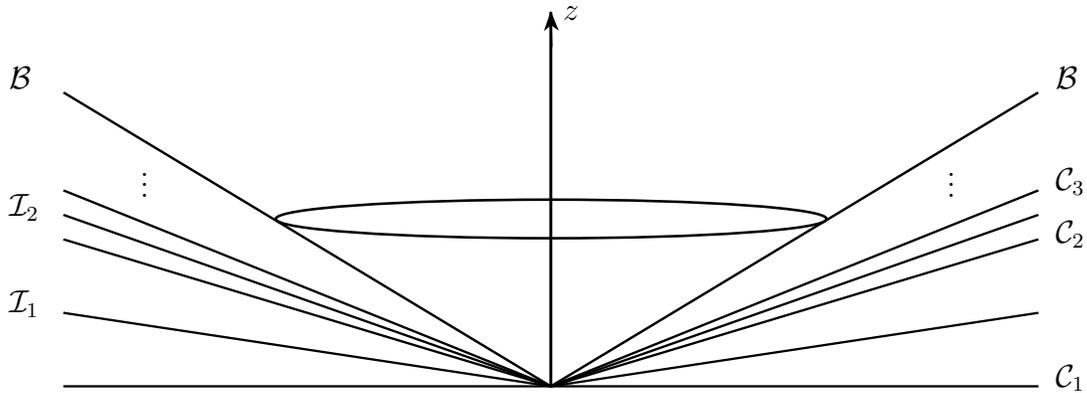
\caption{The main and the intermediate conical surfaces in $\mathbb{R}^{3}_{+}$.}
\label{Cone-Reg}
\end{figure}
We define canonical planar flows $H_{p_k}$ in each conical surface by adding the equation $\theta=\hbox{const}$, where $\theta\in[0,\frac{\pi}{2})$, to the system (\ref{system}). 
\subsection{Definition of the map realizing any given sequence of Dold's coefficients [formula \ref{map}]}\label{constr}
Throughout this subsection, we operate in the special version of conical coordinates $(x,\theta)$, where $x$ denotes a (two-dimensional) coordinate on a conical surface for a given angle $\theta$. Let us define three different maps $f_{1,2,3}\colon \mathbb{R}^{3}_{+}\to \mathbb{R}^{3}_{+}$  by the following formulas: \begin{align}\begin{split}
f_{1}(x,\theta)&=(\tau(\theta)x,\theta),\ \ f_{2}(x,\theta)=(x,\mu(\theta)), \\
\ \ f_{3}(x,\theta)&=\begin{cases}(H_{p_{k}}(x,\sigma(\theta)),\theta) \ &\hbox{for}\ \theta \in[\theta_{k},\widetilde{\theta}_{k}],\\ (H_{p_{k+1}}(x,\sigma(\theta)),\theta) \ &\hbox{for}\ \theta \in [\widetilde{\theta}_{k},\theta_{k+1}],\end{cases}\end{split}
\end{align}
where $\tau,\sigma\colon [0,\frac{\pi}{2}]\to [0,1]$ and $\mu \colon [0,\frac{\pi}{2}]\to [0,\frac{\pi}{2}]$ are one dimensional maps whose graphs are pictured in Figure \ref{maps}. For the map $\tau$, we choose a small  value $\varepsilon\in (0,\frac{1}{2})$.
\begin{figure}[h!] 
\def\svgwidth{0.45\columnwidth} 
\label{maps}
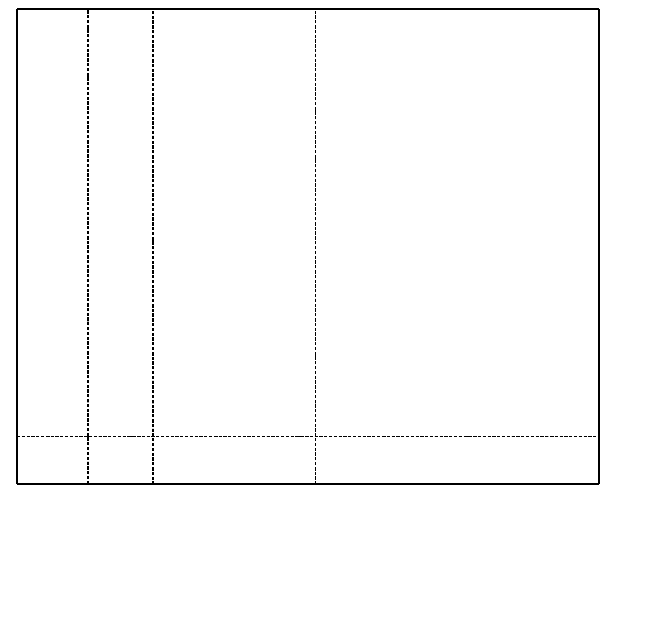
    \centering
    \subfloat{{\def\svgwidth{0.45\columnwidth}
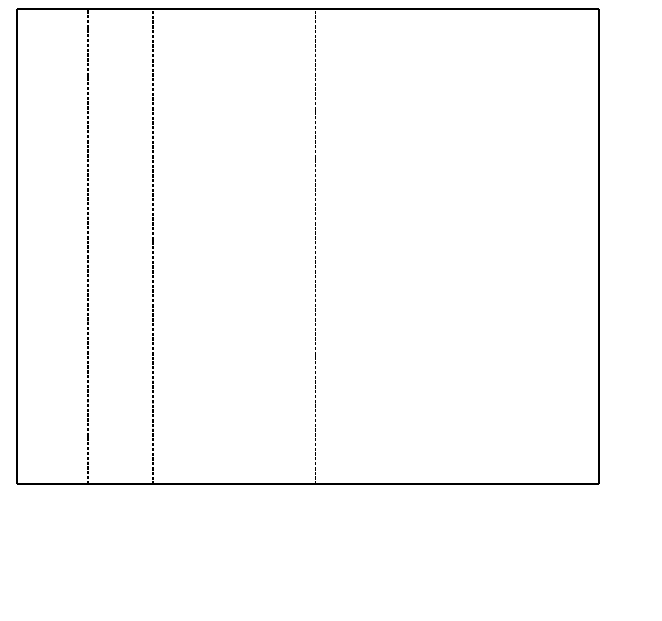}}%
    \qquad
    \subfloat{{\subfloat{{\def\svgwidth{0.45\columnwidth}
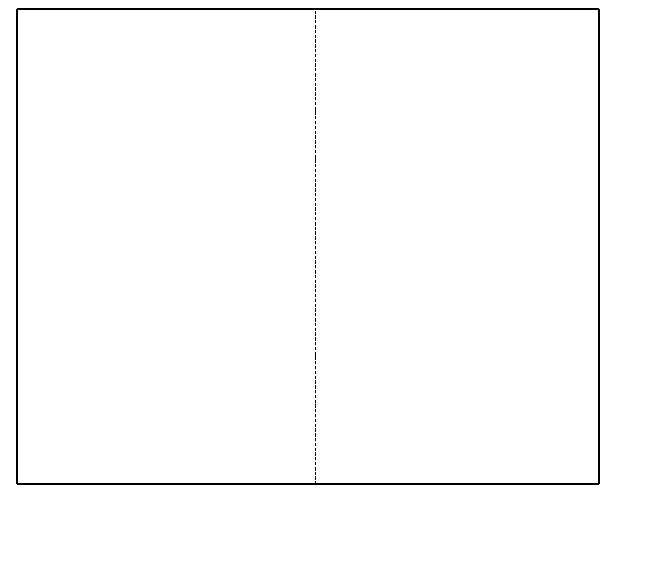}}}}%
\caption{The graphs of maps  $\tau$, $\sigma$ and $\mu$.}
\label{maps}
\end{figure}
\begin{remark}
The brief description of the geometrical meaning of these maps is the following. The  map $f_{1}$  generates a family of sink dynamics in the entire $\mathbb{R}^{3}_{+}$, in which the convergence rate depends on the value of the angle $\theta$. The map $f_{2}$ shifts all intermediate conical surfaces onto the main conical surfaces. The map $f_{3}$  describes the continuous change between discretization of the flows $H_{p}$ for different integers $p$. 
\end{remark}
Now, let $(a_{n})_{n}$ be an  arbitrary integer sequence of Dold's coefficients  and assume that it  has infinitely many non-zero terms, otherwise a sequence $( \ind(f^{n},0))_{n}$ is periodic and the construction proceeds analogously. Let us denote a subsequence of non-zero terms of Dold's coefficients as $(a_{u_{n}})_{n}$.  
To define the desire map, we need a below lemma that allow us to continuously change among different infinite family of rotations.

\vskip 5mm
\begin{lemma}\label{irrational}
For a given irrational number $r\in(0,1)$ there exists a sequence $( s_{n})_{n}=(\frac{v_{n}}{u_{n}})_{n}$ of rational numbers, with $u_{n}$ and $v_{n}$ co-prime, converging to $r$.
\end{lemma}
\begin{proof}
 For sufficiently large $n$ (and all consecutive ones) we choose $\frac{v_{n}}{u_{n}}$ in such a way that $v_{n}$ and $u_{n}$ are co-prime and $\frac{v_{n}}{u_{n}}$ is the largest number of this form less than $r$. Next consider another number $w_{n}>v_{n}$ co-prime with $u_{n}$ such that $\frac{w_{n}}{u_{n}}$ is the smallest number of this form greater than $r$. Therefore, the following inequality \begin{align}
     \left|r-\frac{v_{n}}{u_{n}}\right |\leq \left|\frac{w_{n}}{u_{n}}-\frac{v_{n}}{u_{n}}\right| \leq \frac{j(u_{n})}{u_{n}}
 \end{align}
holds, where $j\colon \mathbb{N}\to \mathbb{N}$ denotes the Jacobsthal function. The function $j(n)$ is defined as the  maximal gap in the list of all the integers co-prime to $n$. In \cite{iwaniec} H. Iwaniec showed that for all natural $n$ this function is bounded from above as follows: \begin{align}
     j(n)\leq \ln^{2}(n).
 \end{align}
 Thus, we can note that \begin{align}
     \lim_{n\to \infty} \frac{j(u_{n})}{u_{n}}=0
 \end{align}
 and hence $\frac{v_{n}}{u_{n}}\to r$ as $n\to \infty$. 
 \end{proof}
Now, we define the rotation map $\bar{O}\colon \mathbb{R}^{3}_{+}\to \mathbb{R}^{3}_{+}$ given by the formula $\bar{O}(r,\varphi,\theta)=(r,\varphi+\psi(\theta),\theta)$, where $\psi\colon [0,\frac{\pi}{2}]\to [0,2\pi]$ is a piecewise-linear map that connects  points $(\theta_{k},\psi(\theta_{k}))$ in the $2$-dimensional graph   of $\psi$ in the following order:$ \ (0,0),(\theta_{1},\frac{2\pi}{d})$, $(\theta_{2},\frac{2\pi}{d})$, $(\theta_{3}, 0)$,  $(\theta_{4}, 2\pi \frac{v_{1}}{u_{1}})$, $(\theta_{5}, 2\pi \frac{v_{2}}{u_{2}})$, $\ldots$ $(\frac{\pi}{4},2\pi s)$, $(\frac{\pi}{2},0)$ (Figure \ref{psi}).

 \begin{figure}[h!] 
\def\svgwidth{0.45\columnwidth} 
\begingroup%
  \makeatletter%
  \providecommand\color[2][]{%
    \errmessage{(Inkscape) Color is used for the text in Inkscape, but the package 'color.sty' is not loaded}%
    \renewcommand\color[2][]{}%
  }%
  \providecommand\transparent[1]{%
    \errmessage{(Inkscape) Transparency is used (non-zero) for the text in Inkscape, but the package 'transparent.sty' is not loaded}%
    \renewcommand\transparent[1]{}%
  }%
  \providecommand\rotatebox[2]{#2}%
  \newcommand*\fsize{\dimexpr\f@size pt\relax}%
  \newcommand*\lineheight[1]{\fontsize{\fsize}{#1\fsize}\selectfont}%
  \ifx\svgwidth\undefined%
    \setlength{\unitlength}{317.48031496bp}%
    \ifx\svgscale\undefined%
      \relax%
    \else%
      \setlength{\unitlength}{\unitlength * \real{\svgscale}}%
    \fi%
  \else%
    \setlength{\unitlength}{\svgwidth}%
  \fi%
  \global\let\svgwidth\undefined%
  \global\let\svgscale\undefined%
  \makeatother%
  \begin{picture}(1,0.84821429)%
    \lineheight{1}%
    \setlength\tabcolsep{0pt}%
    \put(0,0){\includegraphics[width=\unitlength,page=1]{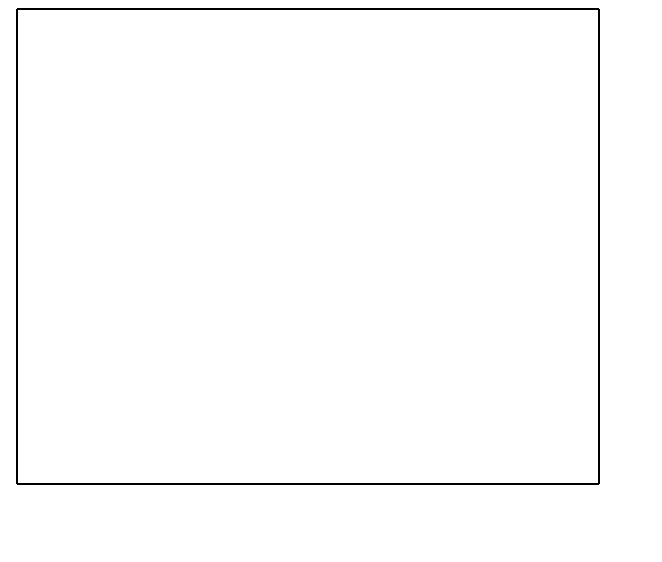}}%
    \put(0.01205357,0.00994315){\makebox(0,0)[lt]{\lineheight{1.25}\smash{\begin{tabular}[t]{l}$0$\end{tabular}}}}%
    \put(0.79368928,0.00994469){\makebox(0,0)[lt]{\lineheight{1.25}\smash{\begin{tabular}[t]{l}$\frac{\pi}{4}$\end{tabular}}}}%
    \put(0.89406176,0.00994469){\makebox(0,0)[lt]{\lineheight{1.25}\smash{\begin{tabular}[t]{l}$\frac{\pi}{2}$\end{tabular}}}}%
    \put(0.02857021,0.77971886){\makebox(0,0)[lt]{\lineheight{1.25}\smash{\begin{tabular}[t]{l}$2\pi s$\end{tabular}}}}%
    \put(0,0){\includegraphics[width=\unitlength,page=2]{psi.pdf}}%
    \put(0.54499367,0.6949617){\makebox(0,0)[lt]{\lineheight{1.25}\smash{\begin{tabular}[t]{l}$\psi$\end{tabular}}}}%
    \put(0,0){\includegraphics[width=\unitlength,page=3]{psi.pdf}}%
  \end{picture}%
\endgroup%

\caption{The example of graph of $\psi$.}
\label{psi}
\end{figure}

  Finally,   we define a boundary-preserving homeomorphism of $\mathbb{R}^{3}_{+}$ as the following composition:
    
    \begin{equation}\label{map}
    f \colon \mathbb{R}^{3}_{+}\to \mathbb{R}^{3}_{+},\ f=\bar{O}\circ f_{1}\circ f_{2}\circ f_{3}.
\end{equation}
We will show in the forthcoming part of the paper that $f$ satisfies the equalities (\ref{31}).
\subsection{Computation of indices}
In this subsection, we prove two lemmas related to the computations of indices of iterates, that will be used in the proof of Theorem \ref{thm1}.
Foremost, we remind the theorem which relates the indices of a map defined on the given space  with indices of its decomposition.
\begin{theorem}[Mayer-Vietoris formula, cf.~ \cite{BaWo}]\label{May-Viet} {\normalfont Let $Y=A\cup B$ be a topological space, $p\in A\cup B$ and $U$ be an open neighborhood of $p$ in $Y$. Let $F\colon U\rightarrow Y$, $F(U\cap A)\subset A$, $F(U\cap B)\subset B$. If $p$ is an isolated fixed point of $F$ and $U$, $U\cap A$, $U\cap B$ and $U\cap A \cap B$ are ENR's, then: \begin{align}
\ind(F,p)= \ind(F_{A},p)+\ind(F_{B},p)-\ind(F_{A\cap B},p), \label{May-viet1}
\end{align}
where $F_{X}$ denotes $F$ restricted to $X$.}
\end{theorem}
Let $k\geq 1$. We define the following regions: \begin{align}\begin{split}
&\mathcal{F}_{1}=\lbrace (r,\varphi,\theta)\mathbin{:}\theta \in[0,\widetilde{\theta}_{1}]\rbrace, \ \mathcal{F}_{k+1}=\lbrace (r,\varphi,\theta)\mathbin{:}\theta\in [\widetilde{\theta}_{k},\widetilde{\theta}_{k+1}]\rbrace,\ \mathcal{G}_{l}=\bigcup_{k=1}^{l}\mathcal{F}_{k}, \\
&\mathcal{G}=\lbrace (r,\varphi,\theta)\mathbin{:}\theta\in [0,\pi/4]\rbrace, \ \mathcal{H}=\lbrace (r,\varphi,\theta)\mathbin{:} \theta\in[\pi/4,\pi/2]\rbrace.\end{split} \end{align}
\begin{lemma}\label{Retr-Lem} The equality $\ind(f^{n}_{\mathcal{F}_{k}},0)=\ind(f^{n}_{\mathcal{C}_{k}},0)$ is satisfied for each $k\geq 1$.
\end{lemma}
\begin{proof}
Let $U$ be an open neighborhood of $0$ in $\mathcal{F}_{k}$. Consider the homotopy $H^{(n)} \colon U\times I\to \mathcal{F}_{k}$ given by the following formula: $H_{t}^{(n)}(r,\varphi,\theta)=f^{n}(r,\varphi, (1-t)\theta+t\theta_{k})$.  Note that this homotopy connects maps $f^{n}$ and $f^{n}\circ R$ in $\mathcal{F}_{k}$, where $R(r,\varphi,\theta)=(r,\varphi,\theta_{k})$ is the retraction map from $\mathcal{F}_{k}$ to $\mathcal{C}_{k}$. It remains to observe that $0$ is the only fixed point of $H_{t}^{(n)}$ for each $t\in I$. Therefore, due to the homotopy invariant property of the fixed point index and the definition of index,  we obtain a sequence of equalities:\begin{align}\ind(f^{n}_{\mathcal{F}_{k}},0)=\ind(H_{0}^{(n)},0)=\ind(H_{1}^{(n)},0)=\ind((f^{n}_{\mathcal{F}_{k}}\circ R),0)=\ind(f^{n}_{\mathcal{C}_{k}},0).\end{align}
\end{proof}
\begin{lemma}\label{induction}
The following equality: \begin{align}\label{lem} \ind(f^{n}_{\mathcal{G}_{l}},0)=\sum_{k=1}^{l}\ind(f^{n}_{\mathcal{C}_{k}},0)-\sum_{k=1}^{l-1}\ind(f^{n}_{\mathcal{I}_{k}},0)\end{align}
is satisfied for all $l\geq 2$.
\end{lemma}
\begin{proof}
The proof is inductive in respect to an integer $l>0$ and is based on the formula \ref{May-viet1} and Lemma $\ref{Retr-Lem}$. Let us first fix $l=2$, hence $\mathcal{G}_{2}=\mathcal{F}_{1}\cup \mathcal{F}_{2}$ and $\ind(f^{n}_{\mathcal{G}_{2}},0)=\ind(f^{n}_{\mathcal{C}_{1}},0)+\ind(f^{n}_{\mathcal{C}_{2}},0)-\ind(f^{n}_{\mathcal{I}_{1}},0)$. Now, we can assume that the equality (\ref{lem}) is satisfied for a fixed $l>2$ and consider the value of $\ind(f^{n}_{\mathcal{G}_{l+1}},0)$. Decomposing the region $\mathcal{G}_{l+1}$ into $\mathcal{G}_{l}\cup \mathcal{F}_{l+1}$ and using the induction step and the formula \ref{May-viet1} once again, we get the final result. 
\end{proof}
\begin{remark}
    Let us note that $\ind(f^{n}_{\mathcal{I}_{k}},0)=\reg_{1}(n)$ for each integer $k>0$, since $f^{n}_{\mathcal{I}_{k}}$ is a sink.
\end{remark}
\subsection{Proof of Theorem \ref{thm1}}\label{proofThm1}
First, we define integers $p_{k}$ assigned to the map $f_{3}$ in the following way. Let $p_{1}=1+\bar{a}_{d}d$ and $p_{2}=1-\bar{a}_{d}d$ for a fixed integer $d>0$, $p_{3}=0$ and  $p_{k+3}=1+a_{u_{k}}u_{k}$ for $k\geq 1$. Applying Lemma $\ref{induction}$, we calculate that: \begin{align}\ind(f^{n}_{\mathcal{G}_{l}},0)=\sum_{k=1}^{l}a_{u_{k}}\reg_{u_{k}}(n)\end{align}
and \begin{align}\ind(f^{n}_{\mathcal{G}},0)=\sum_{k=1}^{\infty}a_{u_{k}}\reg_{u_{k}}(n)\end{align} as $l\to \infty$. Since all Dold's coefficients $a_{k}$ other than $a_{u_{k}}$ are equal to zero, we can write without loss of generality that: \begin{align}\label{coeff}\ind(f^{n}_{\mathcal{G}},0)=\sum_{k=1}^{\infty}a_{k}\reg_{k}(n).\end{align} It remains to note that $\ind(f^{n}_{\mathcal{H}},0)=\reg_{1}(n)$ because $f^{n}_{\mathcal{H}}$ is homotopic to the map $w(r,\theta,\phi)=(\frac{1}{2}r,\theta,\phi)$ by the standard linear homotopy $H_{t}(r,\varphi,\theta)=(1-t)f^{n}(r,\varphi,\theta)+t(\frac{1}{2}r,\varphi,\theta)$. Note that $H_{t}$ does not have any other fixed points than $0$. Thus: \begin{align}
\ind(f^{n},0)&=\ind(f^{n}_{\mathcal{G}\cup\mathcal{H}},0)=\ind(f^{n}_{\mathcal{G}},0)+\ind(f^{n}_{\mathcal{H}},0)-\reg_{1}(n)=\sum_{k=1}^{\infty}a_{k}\reg_{k}(n).
\end{align} Finally, observe that by our construction we have: \begin{align}\ind(\bar{f}^{n},0)=\ind(f^{n}_{\mathcal{C}_{1}},0)=\reg_{1}(n)+\bar{a}_{d}\reg_{d}(n),\end{align} which completes the proof of Theorem \ref{thm1}.
\vskip 5mm
As in  $\mathbb{R}^{4}_{+}$ and in higher dimensions there is a lot of room for the analogous constructions, we may repeat it, realizing in $\partial \mathbb{R}^{4}_{+} \approx \mathbb{R}^{3} $ an arbitrary sequence of Dold's coefficients.
\begin{corollary}\label{Cor-boundR^m}
        For any sequences $(a_{n})_{n=1}^{\infty}$ and $(\bar{a}_{n})_{n=1}^{\infty}$ of Dold's coefficients, there exists a boundary-preserving homeomorphism $f$ of $\mathbb{R}^{m}_{+}$ for $m>3$ such that: 
    \begin{align}
        \ind(f^{n},0)=\sum_{k=1}^{\infty} a_{k}\reg_{k}(n), \ \ \ \ \ind(\bar{f}^{n},0)=\sum_{k=1}^{\infty} \bar{a}_{k}\reg_{k}(n).
    \end{align}
\end{corollary}

\section{Orientation-reversing homeomorphisms \label{sec4}}
We start this section with some basic definitions and known results. We assume that $f$ is a self-homeomorphism of $\mathbb{R}^{m}$ and $p$ an isolated fixed point of $f$.
 \begin{definition} 
\label{accumulated} We say that $p$ is a non-accumulated fixed point of $f$ if there exists a neighborhood $U$ of a point $p$ such that $f^{n}(U)\cap \Per(f)=\lbrace p\rbrace$ for each $n$.
 \end{definition}
\begin{definition}\label{inv}
 We say that the set $\lbrace p\rbrace$ is an isolated invariant set with respect to $f$ if there is a neighborhood $U$ of a point $p$ such that: 
 \begin{align}\label{invariant}
        \bigcap_{k\in \mathbb{Z}}f^{k}(U)=\lbrace p\rbrace.
    \end{align}
\end{definition}
The condition (\ref{invariant}) plays an important role, as it makes it possible to apply the Conley index
 theory to study the dynamics near $p$.
 
Now we remind the result Szymczak et al. \cite{Szymczak} that we will use in the further part of the paper.
 Let $(X,d)$ be a locally compact metric space and $f\colon X\to X$ a continuous map. \begin{definition}
We define a sequence $\sigma_{x} \colon \mathbb{Z}_{-}\to X$ as a left solution of $f$ through $x$ if $\sigma(0)=x$ and $f(\sigma(n-1))=\sigma(n)$ for each $n\in \mathbb{Z}_{-}$.
 \end{definition}
 \begin{definition}
    We define the stable and unstable sets of the compact invariant set $S$ as follows: \begin{align}
        \begin{split}
            W^{s}(S,f)&=\lbrace x\in X \mathbin{:} \lim_{n\to \infty}d(f^{n}(x),S)=0\rbrace,\\
            W^{u}(S,f)&= \lbrace x\in X \mathbin{:} \exists{\sigma_{x}} \ \hbox{such that}\ \lim_{n\to \infty} d(\sigma(-n),S)=0\rbrace.
        \end{split}
    \end{align}
 \end{definition}
 \begin{theorem}\label{szymczakthm}{\normalfont(cf.~ \cite{Szymczak})}. Suppose that $f\colon \mathbb{R}^{m}_{+}\to \mathbb{R}^{m}_{+}$ is a continuous map and $S\subset \partial\mathbb{R}^{m}_{+}$ is an isolated invariant set with respect to $f$. Then \begin{align}
     \ind(f,S)=\begin{cases}
         \ind(\restr{f}{\partial\mathbb{R}^{m}_{+}},S) &\hbox{if} \ \ W^{u}(S,f)\subset \partial\mathbb{R}^{m}_{+},\\
         0  &\hbox{if} \ \ W^{s}(S,f)\subset \partial\mathbb{R}^{m}_{+}.
     \end{cases}
 \end{align}
 \end{theorem}
 
The theorem that inspired our considerations is related to orientation-reversing homeomorphisms in dimension $3$ with a fixed point being isolated as an invariant set. 
\begin{theorem}\label{revers}{\normalfont(cf.~ \cite{Corbato}, Theorem C)}. For any sequence $(a_{n})_{n=1}^{\infty}$ of Dold's coefficients, there is an orientation-reversing homeomorphism $f$ of $\mathbb{R}^{3}$ such that $\lbrace 0\rbrace$ is an isolated invariant set and \begin{align}
    \ind(f^{n},0)=\sum_{k=1}^{\infty}a_{k}\reg_{k}(n)
\end{align}
if and only if \begin{enumerate}
    \item there are finitely many non-zero $a_{k}$,
    \item $a_{1}\leq 1$ and $a_{k}\leq 0$ for all odd $k>1$.
\end{enumerate}
\end{theorem}
\begin{remark}
The following question arises concerning Theorem \ref{revers}$\mathbin{.}$ Rejecting the assumption that $\lbrace 0\rbrace$ is necessarily an isolated invariant set of $f$, what are possible forms of fixed point indices of iterates? We give the complete answer to this question in Theorem \ref{thm3} below.
\end{remark}

\begin{theorem}\label{thm3}
    For any sequence $( a_{n})_{n=1}^{\infty}$ of Dold's coefficients, there is an orientation-reversing homeomorphism $f$ of $\mathbb{R}^{m}$ ($m\geq 3$) such that: \begin{align} \label{53}
        \ind(f^{n},0)=\sum_{k=1}^{\infty}a_{k}\reg_{k}(n).
    \end{align}
\end{theorem}
\section{Constructions needed in the proof of Theorem  (\ref{thm3}) \label{sec5}}
In this section, we construct a map $f$, an orientation-reversing self-homeomorphism of $\mathbb{R}^{3}$, which realizes any sequence of Dold's coefficients as a sequence of fixed point indices of iterates. 
\subsection{Model of hyperbolic and elliptic dynamics}
Let us consider a planar sector $\mathcal{S}=\lbrace (r,\varphi): 0\leq \alpha-2\beta\leq \varphi \leq \alpha+2\beta\leq 2\pi\rbrace$. We split it into the three following sectors (Figure \ref{model}): \begin{align}\begin{split}
\mathcal{I}_{R}&=\lbrace (r,\varphi)\mathbin{:} \alpha-2\beta \leq \varphi\leq \alpha-\beta\rbrace\ \hbox{(right)},\\
\mathcal{I}_{C}&=\lbrace (r,\varphi)\mathbin{:} \alpha-\beta\leq \varphi\leq \alpha+\beta\rbrace \ \hbox{(center)},\\ 
\mathcal{I}_{L}&=\lbrace (r,\varphi)\mathbin{:} \alpha+2\beta \leq \varphi\leq \alpha+\beta\rbrace \ \hbox{(left)}. \end{split}
\end{align}
\begin{figure}[!h] 
\def\svgwidth{0.75\columnwidth}
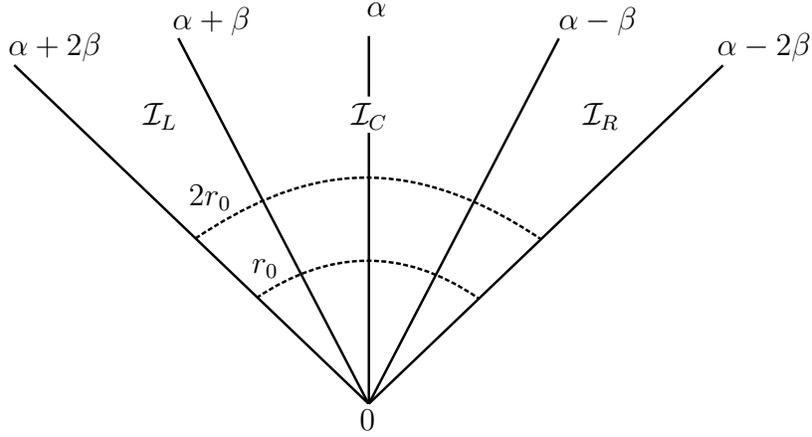
\caption{The splitting of the sector $\mathcal{S}$.}
\label{model}
\end{figure}
Now, we define the flow $G\colon \mathcal{S}\times \mathbb{R}\to \mathcal{S}$ by the following systems of equations:
\begin{align}\label{conical}\begin{split}
\mathcal{I}_{R}&: \ \begin{cases}
    \dot{r}(t)=-rs-r(\gamma(r)+\lambda(r))(1-s),\\
    \varphi_{s}=\alpha-\beta-s\beta,
\end{cases} \\
\mathcal{I}_{C}&: \ \begin{cases}
    \dot{r}(t)=-\gamma(r) r\cos(\pm \pi(\frac{\varphi}{\beta}+\frac{\beta-\alpha}{\beta}))-r\lambda(r),\\
    \dot{\varphi}(t)=-\gamma(r) \sin(\pm \pi(\frac{\varphi}{\beta}+\frac{\beta-\alpha}{\beta})),\end{cases} \ \ s\in[0,1] \\
    \mathcal{I}_{L}&: \ \begin{cases}
        \dot{r}(t)=-rs-r(\gamma(r)+\lambda(r))(1-s),\\
    \varphi_{s}=\alpha+\beta+s\beta,
\end{cases}
\end{split}
\end{align}
where $\lambda,\gamma\colon [0,\infty)\to [0,1]$ are the  maps whose graphs are shown in Figure \ref{twomaps}. Note that the choice between signs ($+$) or ($-$) in the arguments of trigonometric functions given in (\ref{conical}) denotes the choice of generating elliptic or  hyperbolic dynamics (each containing two sectors of them) in the phase space $\mathcal{I}_{C}$, respectively.
In the above sectors, the equations (\ref{conical}) describe the  continuous  change  between hyperbolic or elliptic dynamics and the sink. Moreover, increasing $r$ we get that the flow ultimately becomes the identity at $r_0$ and a sink at $2r_0$.
\begin{figure}[h!]
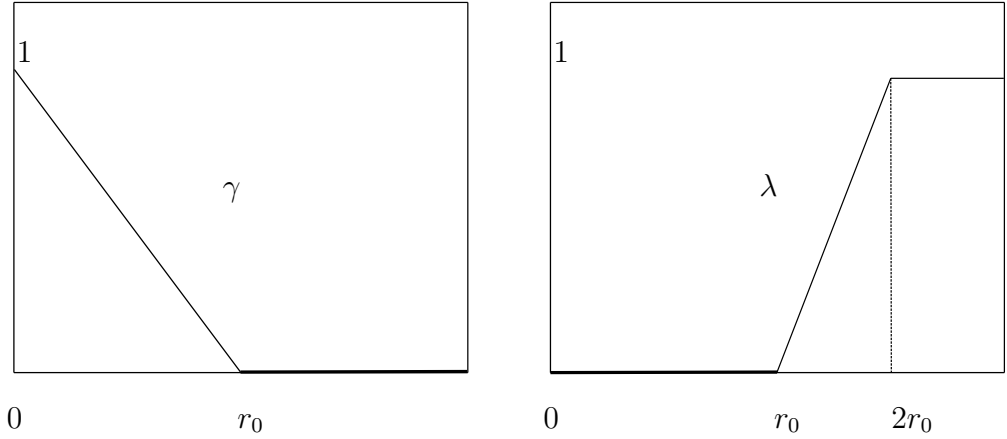
 
\def\svgwidth{0.45\columnwidth} 
\begingroup%
  \makeatletter%
  \providecommand\color[2][]{%
    \errmessage{(Inkscape) Color is used for the text in Inkscape, but the package 'color.sty' is not loaded}%
    \renewcommand\color[2][]{}%
  }%
  \providecommand\transparent[1]{%
    \errmessage{(Inkscape) Transparency is used (non-zero) for the text in Inkscape, but the package 'transparent.sty' is not loaded}%
    \renewcommand\transparent[1]{}%
  }%
  \providecommand\rotatebox[2]{#2}%
  \newcommand*\fsize{\dimexpr\f@size pt\relax}%
  \newcommand*\lineheight[1]{\fontsize{\fsize}{#1\fsize}\selectfont}%
  \ifx\svgwidth\undefined%
    \setlength{\unitlength}{317.48031496bp}%
    \ifx\svgscale\undefined%
      \relax%
    \else%
      \setlength{\unitlength}{\unitlength * \real{\svgscale}}%
    \fi%
  \else%
    \setlength{\unitlength}{\svgwidth}%
  \fi%
  \global\let\svgwidth\undefined%
  \global\let\svgscale\undefined%
  \makeatother%
  \begin{picture}(1,0.84821429)%
    \lineheight{1}%
    \setlength\tabcolsep{0pt}%
    \put(0,0){\includegraphics[width=\unitlength,page=1]{0gammamap.pdf}}%
    \put(0.01205357,0.00994315){\makebox(0,0)[lt]{\lineheight{1.25}\smash{\begin{tabular}[t]{l}$0$\end{tabular}}}}%
    \put(0.45933986,0.00994469){\makebox(0,0)[lt]{\lineheight{1.25}\smash{\begin{tabular}[t]{l}$r_{0}$\end{tabular}}}}%
    \put(0.03165746,0.71914927){\makebox(0,0)[lt]{\lineheight{1.25}\smash{\begin{tabular}[t]{l}$1$\end{tabular}}}}%
    \put(0,0){\includegraphics[width=\unitlength,page=2]{0gammamap.pdf}}%
    \put(0.43102843,0.45980161){\makebox(0,0)[lt]{\lineheight{1.25}\smash{\begin{tabular}[t]{l}$\gamma$\end{tabular}}}}%
    \put(0,0){\includegraphics[width=\unitlength,page=3]{0gammamap.pdf}}%
  \end{picture}%
\endgroup%

    \centering
    \subfloat{{\def\svgwidth{0.45\columnwidth}
\begingroup%
  \makeatletter%
  \providecommand\color[2][]{%
    \errmessage{(Inkscape) Color is used for the text in Inkscape, but the package 'color.sty' is not loaded}%
    \renewcommand\color[2][]{}%
  }%
  \providecommand\transparent[1]{%
    \errmessage{(Inkscape) Transparency is used (non-zero) for the text in Inkscape, but the package 'transparent.sty' is not loaded}%
    \renewcommand\transparent[1]{}%
  }%
  \providecommand\rotatebox[2]{#2}%
  \newcommand*\fsize{\dimexpr\f@size pt\relax}%
  \newcommand*\lineheight[1]{\fontsize{\fsize}{#1\fsize}\selectfont}%
  \ifx\svgwidth\undefined%
    \setlength{\unitlength}{317.48031496bp}%
    \ifx\svgscale\undefined%
      \relax%
    \else%
      \setlength{\unitlength}{\unitlength * \real{\svgscale}}%
    \fi%
  \else%
    \setlength{\unitlength}{\svgwidth}%
  \fi%
  \global\let\svgwidth\undefined%
  \global\let\svgscale\undefined%
  \makeatother%
  \begin{picture}(1,0.84821429)%
    \lineheight{1}%
    \setlength\tabcolsep{0pt}%
    \put(0,0){\includegraphics[width=\unitlength,page=1]{lambdamap.pdf}}%
    \put(0.01205357,0.00994315){\makebox(0,0)[lt]{\lineheight{1.25}\smash{\begin{tabular}[t]{l}$0$\end{tabular}}}}%
    \put(0.45933986,0.00994469){\makebox(0,0)[lt]{\lineheight{1.25}\smash{\begin{tabular}[t]{l}$r_{0}$\end{tabular}}}}%
    \put(0.68676485,0.00994469){\makebox(0,0)[lt]{\lineheight{1.25}\smash{\begin{tabular}[t]{l}$2r_{0}$\end{tabular}}}}%
    \put(0.03165746,0.71914927){\makebox(0,0)[lt]{\lineheight{1.25}\smash{\begin{tabular}[t]{l}$1$\end{tabular}}}}%
    \put(0,0){\includegraphics[width=\unitlength,page=2]{lambdamap.pdf}}%
    \put(0.43102843,0.45980161){\makebox(0,0)[lt]{\lineheight{1.25}\smash{\begin{tabular}[t]{l}$\lambda$\end{tabular}}}}%
    \put(0,0){\includegraphics[width=\unitlength,page=3]{lambdamap.pdf}}%
  \end{picture}%
\endgroup%
}}%
\caption{The graphs of maps $\lambda$ and $\gamma$ with a fixed value $r_{0}\in (0,+\infty)$.}
\label{twomaps}
\end{figure}
\begin{remark} \label{rm61}
 For each angle $\varphi\in \lbrace \alpha-2\beta,\alpha-\beta, \alpha+\beta,\alpha+2\beta\rbrace$ we have the same sink generated by the equation $\dot{r}=-r$. 
\end{remark}
Now, we transfer the dynamics of (\ref{conical}) into $\mathbb{R}^{3}$. Let $\breve{\mathcal{S}}\subset \mathbb{R}^{3}$ be a solid cone that is formed after the rotation of $\mathcal{S}$ around its central ray determined by the angle $\varphi=\alpha$. We introduce a special version of conical coordinate system $(r,\rho_{\alpha},\varphi)$ in $\breve{\mathcal{S}}$, where $r$ and $\varphi$ represent the standard polar coordinates and $\rho_{\alpha} \in [0,\pi]$ denotes the coordinate that rotates the points around a central ray of a cone. 

We define the flow $\breve{G}\colon \breve{\mathcal{S}}\times \mathbb{R}\to \breve{\mathcal{S}}$  by the following system of equations: \begin{align}\label{newflow}
\breve{G}=\begin{cases}
        G,\\
        \dot{\rho}_{\alpha}=0,
    \end{cases}
\end{align}
and let $\breve{g}_{h}$ and $\breve{g}_{e}$ denote the time-one maps of $\breve{G}$ depending on the choice of hyperbolic or elliptic dynamics, respectively (Figure \ref{dynamics}).
\begin{figure}[!h] 
\def\svgwidth{1\columnwidth}
\begingroup%
  \makeatletter%
  \providecommand\color[2][]{%
    \errmessage{(Inkscape) Color is used for the text in Inkscape, but the package 'color.sty' is not loaded}%
    \renewcommand\color[2][]{}%
  }%
  \providecommand\transparent[1]{%
    \errmessage{(Inkscape) Transparency is used (non-zero) for the text in Inkscape, but the package 'transparent.sty' is not loaded}%
    \renewcommand\transparent[1]{}%
  }%
  \providecommand\rotatebox[2]{#2}%
  \newcommand*\fsize{\dimexpr\f@size pt\relax}%
  \newcommand*\lineheight[1]{\fontsize{\fsize}{#1\fsize}\selectfont}%
  \ifx\svgwidth\undefined%
    \setlength{\unitlength}{352.5bp}%
    \ifx\svgscale\undefined%
      \relax%
    \else%
      \setlength{\unitlength}{\unitlength * \real{\svgscale}}%
    \fi%
  \else%
    \setlength{\unitlength}{\svgwidth}%
  \fi%
  \global\let\svgwidth\undefined%
  \global\let\svgscale\undefined%
  \makeatother%
  \begin{picture}(1,0.36170213)%
    \lineheight{1}%
    \setlength\tabcolsep{0pt}%
    \put(0,0){\includegraphics[width=\unitlength,page=1]{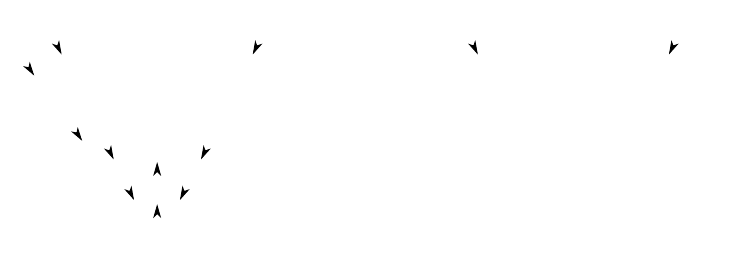}}%
    \put(0.20815629,0.00497372){\makebox(0,0)[lt]{\lineheight{1.25}\smash{\begin{tabular}[t]{l}$0$\end{tabular}}}}%
    \put(0.77490637,0.00497372){\makebox(0,0)[lt]{\lineheight{1.25}\smash{\begin{tabular}[t]{l}$0$\end{tabular}}}}%
    \put(0,0){\includegraphics[width=\unitlength,page=2]{FiniteCones.pdf}}%
    \put(0.22293492,0.27351898){\makebox(0,0)[lt]{\lineheight{1.25}\smash{\begin{tabular}[t]{l}$r=r_{0}$\end{tabular}}}}%
    \put(0.78880972,0.2735187){\makebox(0,0)[lt]{\lineheight{1.25}\smash{\begin{tabular}[t]{l}$r=r_{0}$\end{tabular}}}}%
    \put(0,0){\includegraphics[width=\unitlength,page=3]{FiniteCones.pdf}}%
    \put(0.38939297,0.20812474){\makebox(0,0)[lt]{\lineheight{1.25}\smash{\begin{tabular}[t]{l}$\breve{g}_{h}$\end{tabular}}}}%
    \put(0.95035868,0.20812518){\makebox(0,0)[lt]{\lineheight{1.25}\smash{\begin{tabular}[t]{l}$\breve{g}_{e}$\end{tabular}}}}%
  \end{picture}%
\endgroup%

\caption{The dynamics in $\breve{\mathcal{S}}$ induced by (\ref{conical}). The disks in gray represent stationary points of $\breve{G}$.} 
\label{dynamics}
\end{figure}
Our first objective is to determine the following values.

\begin{lemma}\label{ind-values}
$\ind(\breve{g}_{h},0)=0$ and 
$\ind(\breve{g_{e}},0)=0$. 
\end{lemma}
\begin{proof}
     The fact that $\ind(\breve{g}_{h},0)=0$ results straightforwardly from Theorem 2 in \cite{Szymczak} (Theorem \ref{szymczakthm} above).

In order to show that  $\ind(\breve{g}_{e},0)=0$, let us consider a retraction $R\colon \mathbb{R}^{3}\to \breve{\mathcal{S}}$ and the map $\bar{g}=i\circ \breve{g}_{e}\circ R\colon \mathbb{R}^{3}\to \mathbb{R}^{3}$, where $i\colon \breve{\mathcal{S}}\to \mathbb{R}^{3}$ denotes an inclusion of $\breve{\mathcal{S}}$ in $\mathbb{R}^{3}$. The fixed point index of such a map is equal to the degree of the associated map $\xi\colon S^{2}\to S^{2}$ given by the following formula: \begin{align}
     \xi(p)=\varepsilon \frac{p-\bar{g}(p)}{|p-\bar{g}(p)|},
\end{align} where $S^{2}$ is a boundary of an open three-dimensional ball $B(0,\varepsilon)$ with a sufficiently small radius $\varepsilon<r_{0}$.
Let us observe that the point $S=(0,0,-\varepsilon)$ is a regular value of $\xi$ and $\xi^{-1}(\lbrace S\rbrace )=\lbrace S,N\rbrace$, where $N$ is the North Pole of $S^{2}$ with the coordinates $(0,0,\varepsilon)$. 
Without loss of generality, we show the calculation in the case when $\varepsilon=1$. Note that the map $\xi$ is given by the following formula: \begin{align}
    \xi(\varphi,\theta)=\begin{cases}
        (\varphi,\theta) \  &\hbox{for $\theta\in [-\frac{\pi}{2},0]$}, \\
        (\varphi,3\theta)  \ &\hbox{for $\theta\in [0,\frac{\pi}{2}].$}
    \end{cases}
\end{align}
Transferring this map to Cartesian coordinates $p=(x,y,z)$, we get the following: \begin{align}
    \xi(x,y,z)=
        \begin{cases}
            (x,y,z) \ &\hbox{for $z\in [-1,0]$},\\ 
            ((4(x^{2}+y^{2})-3)x,(4(x^{2}+y^{2})-3)y,3z-4z^{3}) \ &\hbox{for $z\in (0,1]$}.
        \end{cases}
\end{align}
It remains to note that in fact $\xi^{-1}(\lbrace S\rbrace)=\lbrace S, N\rbrace$ and determine the local derivative $D\xi$ at these points. It is obvious that $D\xi_{S}=I_{2}$ since $\xi$ is the identity map at $S$ and thus $\mathrm{sign}({\det D\xi_{S}})=1$. Let us now determine the sign of $\det(D\xi_{N})$. We consider the tangent map $D\xi_{N}\colon T_{N}S^{2}\to T_{S}S^{2}$ with compatible bases $\mathcal{B}(T_{N}S^{2})=\lbrace [1,0,0], [0,1,0]\rbrace$ and $\mathcal{B}(T_{S}S^{2})=\lbrace [1,0,0], [0,-1,0]\rbrace$ that determine the orientation of the sphere. Taking the transformation of vectors from $T_{N}S^{2}$ into vectors in $T_{S}S^{2}$, we get the following result:\begin{align}
    D\xi_{N}=\left[\begin{array}{cc}-3&0\\0&3\end{array}\right].
\end{align}
Hence, we have $\mathrm{sign}(\det D\xi_{N})=-1$ and thus the map $\xi$ reverses the orientation of $S^{2}$ at $N$. Finally, $\deg(\xi,S^{2})=0$.
\end{proof}

\subsection{Periodic dynamics}
 In this section, we describe some special periodic dynamics, which we will use in  the  construction of our  map. First, let us define $\vee_{i=1}^{k}\breve{\mathcal{S}}_{i}$, a bouquet of $k$ solid cones with a family of associated flows $\breve{G}_{i}\colon \breve{\mathcal{S}}_{i}\times \mathbb{R}\to \breve{\mathcal{S}}_{i}$ (cf. \ref{newflow}). We require that the value of $r_{0}$ is the same for all flows in a given bouquet (see Figure \ref{dynamics}). Now, we place such a bouquet (for appropriately selected values of $\alpha$ and $\beta$) into the following region: \begin{align}\mathcal{J}_{k}=\lbrace (r,\rho_{\alpha},\varphi): \alpha-\beta \leq \varphi \leq \alpha+\beta\rbrace, \ \ \vee_{i=1}^{k}\breve{\mathcal{S}}_{i} \subset \mathcal{J}_{k}.\end{align} Note that $\mathcal{J}_{k}$ is a solid cone positioned symmetrically with respect to the plane $z=0$. The elements of the bouquet also lie symmetrically with respect to this plane. Let us now define the flow $S\colon \mathbb{R}^{3}\times \mathbb{R}\to \mathbb{R}^{3}$ given by $\dot{r}=-r$, and the map $G_{k}\colon \mathcal{J}_{k}\to \mathcal{J}_{k}$ given by the following system of equations: \begin{align}
    G_{k}(p)=\begin{cases}
        \breve{G}_{i}(p,1) \ &\hbox{for $p\in \breve{\mathcal{S}}_{i}$ and $i\in \lbrace 1,\ldots,k\rbrace$},\\
        S(p,1) \ &\hbox{otherwise.}
    \end{cases}
\end{align}

\begin{figure}[!h] 
\def\svgwidth{0.6\columnwidth}
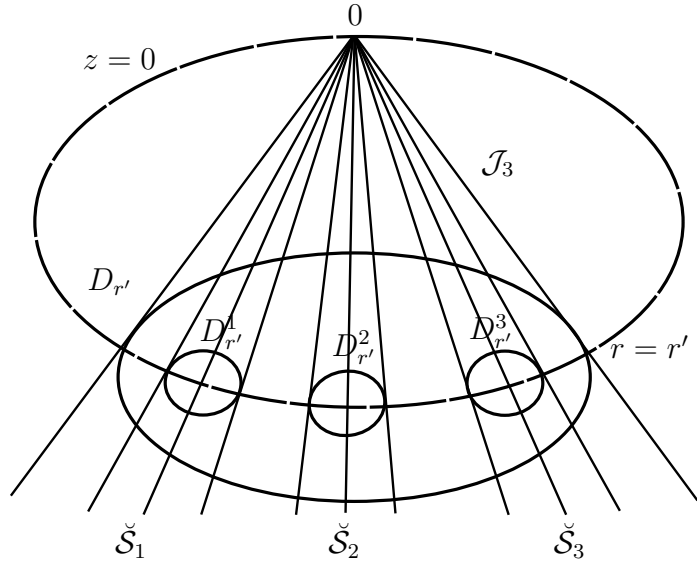
\caption{The solid cone $\mathcal{J}_{3}$ with $D_{r'}$ determined by the value $r=r'$.}
\label{thewedge}
\end{figure}
The map $G_{k}$ restricted to any element of a bouquet generates hyperbolic or elliptic dynamics, and outside the bouquet generates a sink.
Now let us fix the radius $r'>0$ and take  $D_{r'}\subset \mathcal{J}_{k}$, a disk determined by the value of $r'$. Next we consider the family $\lbrace D^{i}_{r'}\rbrace_{i=1}^{k}\subset D_{r'}$ of disks formed by the intersection of $D_{r'}$ with $\vee_{i=1}^{k}\breve{\mathcal{S}}_{i}$ (Figure \ref{thewedge}). 

Our aim  is to define a self-map of $\mathcal{J}_{k}$ that permutes these disks and which would (roughly speaking) provide the contribution to index equal $\reg_k$.  

To implement this idea, let us consider the unit disk $D=\lbrace (r,\varphi)\ |\ r\leq 1\rbrace$ in the plane and a subset of it $A=\lbrace (r,\varphi)\ | \ r\leq \frac{3}{4}\rbrace$. We fix a family $\lbrace D^{i} \rbrace_{i=1}^{k}$ of the copies of the same disk in $A$  with the centers located every $\frac{2\pi}{k}$ degree around the origin, lying on the same circle. Then let us define the map $v\colon D\to D$ given by the  formula $v(r,\varphi)=(r,\varphi+\delta(r))$, where $\delta\colon [0,1]\to [0,\frac{2\pi}{k}]$ is the map given by the formula: \begin{align}
    \delta(r)=\begin{cases}
        \frac{2\pi}{k} \ &\hbox{for $r\in [0,\frac{3}{4}]$},\\
        \frac{8\pi}{k}(1-r) \ &\hbox{for $r\in [\frac{3}{4},1]$.}
    \end{cases}
\end{align}
Now, let us define a homeomorphism $h_{r'}\colon D\to D_{r'}$, which maps the family $\lbrace D^{i}\rbrace_{i=1}^{k}$ onto the family $\lbrace D_{r'}^{i}\rbrace_{i=1}^{k}$ in $D_{r'}$ in such a way that $h_{r'}(D^{i})=D^{i}_{r'}$ (Figure \ref{disks}). Using this homeomorphism and the map $v$, we define the self-map of $D_{r'}$, i.e. $u_{r'}\colon D_{r'}\to D_{r'}$ given by the equation $u_{r'}(x)=h_{r'}\circ v\circ h_{r'}^{-1}(x)$, where $x$ denotes a two-dimensional coordinate of a point in $D_{r'}$. Note that maps $u_{r'}$ and $v$ are topologically conjugated. 
\begin{figure}[!h] 
\def\svgwidth{1\columnwidth}
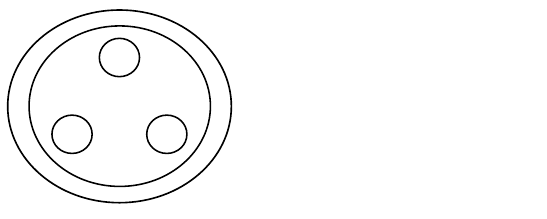
\caption{The example of the topological conjugation of maps $v$ and $u_{r'}$ for $k=3$. The dashed lines represent the intersection of $D_{r'}$ with the plane $z=0$, cf. Figure \ref{thewedge}.}
\label{disks}
\end{figure}
Since the map $v$ is not strictly related to the choice of the value of the radius, we can consider such conjugation for any other value of $r>0$. As a consequence, we can extend the map $u_{r'}$ to the self-map of $\mathcal{J}_{k}$, i.e. $u_{k}\colon \mathcal{J}_{k}\to \mathcal{J}_{k}$ by the following system of equations: \begin{align}u_{k}(r,x)=\label{65}
    \begin{cases}
        h_{r}\circ v\circ h_{r}^{-1}(x) \ &\hbox{for $x\in D_{r}$ and $r>0$,}\\
        0 \ &\hbox{for $r=0$.}
    \end{cases}
\end{align}
\vskip 5mm Finally, we define the main map  of this section, i.e. 
\begin{equation}\label{Theta}
   \Theta_{k}\colon \mathcal{J}_{k}\to \mathcal{J}_{k}, \ \Theta_{k}=u_{k}\circ G_{k}.
\end{equation}
In the next steps of the construction, we will show that this map generates  periodic dynamics in $\mathcal{J}_{k}$, related to the contribution of $\reg_k$'s with appropriate coefficients in the formula for indices of iterates.
\subsection{\label{bouquet} Bouquet of maps}
This short subsection is devoted to the definition of the bouquet consisting of a map of the form (\ref{Theta}). Let us assume for now that $k>1$. The situation in which $k=1$ will be described separately. Let $a_{k}$ be an arbitrary integer and consider a bouquet of $|a_{k}|$ regions of $\mathcal{J}_{k}$, i.e. $\vee_{j=1}^{|a_{k}|}(\mathcal{J}_{k})_{j}$. Now, with such a bouquet, we identify the bouquet of maps  $\vee_{j=1}^{|a_{k}|}(\Theta_{k})_{j}\colon \vee_{j=1}^{|a_{k}|}(\mathcal{J}_{k})_{j}\to \vee_{j=1}^{|a_{k}|}(\mathcal{J}_{k})_{j}$, canonically induced by the following formula:
\begin{align*}
    \vee_{j=1}^{|a_{k}|}(\Theta_{k})_{j}(p)=(\Theta_{k})_{j}(p) \ \hbox{for $p\in(\mathcal{J}_{k})_{j}$ and $j\in \lbrace 1,\ldots,|a_{k}|\rbrace$.} \end{align*}
We introduce a slight change for $k=1$. For $a_{1}>0$ we consider a bouquet of $a_{1}-1$ regions of $\mathcal{J}_{1}$, i.e. $\vee_{j=1}^{a_{1}-1}(\mathcal{J}_{1})_{j}$, and for $a_{1}<0$ a bouquet of $|a_{1}|+1$ regions of $\mathcal{J}_{1}$, i.e. $\vee_{j=1}^{|a_{1}|+1}(\mathcal{J}_{1})_{j}$.
\subsection{\label{64}Odd Dold's coefficients}
Let us consider an arbitrary sequence of Dold's coefficients $\lbrace a_{n}\rbrace_{n=1}^{\infty}$.
We can assume without loss of generality that this sequence has only non-zero terms (see \ref{coeff}). \vskip 5mm 
Let us fix a subsequence $(a_{2n-1})_{n}$ of odd indices and consider the following solid cones, which are formed after the rotation around the central rays lying in the $\mathrm{OXY}$ plane: \begin{align}\mathcal{M}_{2k-1}=\left\lbrace (r,\rho_{\frac{2\pi}{2^{k}}},\varphi)\mathbin{:} \frac{2\pi}{2^{k}}-\frac{2\pi}{2^{k+3}}\leq \varphi \leq \frac{2\pi}{2^{k}}+\frac{2\pi}{2^{k+3}}\right\rbrace.\end{align} Note that these cones are symmetrical with respect to $z=0$ plane and converges to the ray $\mathcal{M}_{\infty}=\lbrace (r,\varphi) \mathbin{:}  \varphi=0 \rbrace$ that lies on the plane $z=0$ (Figure \ref{wedge}). 
\begin{figure}[!h] 
\def\svgwidth{0.65\columnwidth}
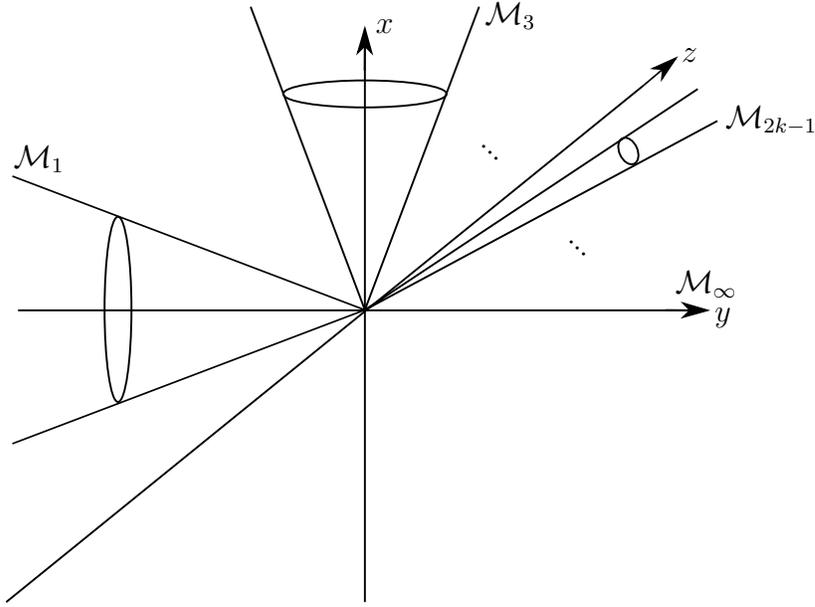
\caption{The graphical illustration of the cones $\mathcal{M}_{2k-1}$ and $\mathcal{M}_{\infty}$.} 
\label{wedge}
\end{figure}
We take a wedge of $|a_{2k-1}|$ solid cones of $\mathcal{J}_{2k-1}$ in each of $\mathcal{M}_{2k-1}$ for $k>1$ and for $k=1$, a wedge of $a_{1}-1$ or $|a_{1}|+1$ cones of $\mathcal{J}_{1}$, depending on the value of $a_{1}$ (see \ref{bouquet}). Now, we define the map $\Theta\colon \vee_{k=1}^{\infty} \mathcal{M}_{2k-1}\to \vee_{k=1}^{\infty}\mathcal{M}_{2k-1}$ by the following formulas: \begin{align}\label{map}
    \Theta(p)=\begin{cases}
    \vee_{j=1}^{|a_{1}|\pm 1}(\Theta_{1})_{j}(p) \ &\hbox{for $p\in \mathcal{M}_{1}\cap (\mathcal{J}_{1})_{j}$ and $j\in \lbrace 1,\ldots, |a_{1}|\pm 1\rbrace$},\\
        \vee_{j=1}^{|a_{2k-1}|}(\Theta_{2k-1})_{j}(p) \ &\hbox{for $p\in \mathcal{M}_{2k-1}\cap (\mathcal{J}_{2k-1})_{j}$, $k>1$ and $j\in \lbrace 1,\ldots, |a_{2k-1}|\rbrace$},\\
        S(p,1) \ &\hbox{for $p\in \mathcal{M}_{\infty}$ and otherwise},
    \end{cases}
\end{align}
 and require that the following conditions are satisfied (remind that $S(p,1)$ is a time-one map of the flow $\dot{r}=-r$).
\begin{itemize}
    \item The value of $r_{0}$ for each map $(G_{2k-1}(p))_{i}$ is equal to $\frac{1}{k}$.
    \item If $a_{2k-1}<0$, then $(G_{2k-1}(p))_{j}$ generates hyperbolic dynamics in each of $\breve{\mathcal{S}}_{i}\subset (\mathcal{J}_{2k-1})_{j}$ (the choice of ($-$) in (\ref{conical})).
    \item If $a_{2k-1}>0$, then $(G_{2k-1}(p))_{j}$ generates elliptic dynamics in each of $\breve{\mathcal{S}}_{i}\subset (\mathcal{J}_{2k-1})_{j}$ (the choice of ($+$) in (\ref{conical})).
\end{itemize}
\subsection{Even Dold's coefficients}
Let us consider the following solid cones, which are formed after the rotation around the central rays lying in the $\mathrm{OXZ}$ plane: \begin{align}\begin{split}
    &\bar{\mathcal{A}}_{+}=\left\lbrace (r,\rho_{\frac{\pi}{2}},\varphi) \mathbin{:} \frac{3\pi}{8}\leq \varphi \leq \frac{5\pi}{8}\right\rbrace, \ \ \mathcal{A}_{+}=\left\lbrace (r,\rho_{\frac{\pi}{2}},\varphi) \mathbin{:} \frac{\pi}{4}\leq \varphi \leq \frac{3\pi}{4}\right\rbrace, \\ &\bar{\mathcal{A}}_{-}=\left\lbrace (r,\rho_{\frac{3\pi}{2}},\varphi) \mathbin{:} \frac{11\pi}{8}\leq \varphi \leq \frac{13\pi}{8}\right\rbrace, \ \ \mathcal{A}_{-}=\left\lbrace (r,\rho_{\frac{3\pi}{2}},\varphi) \mathbin{:} \frac{5\pi}{4}\leq \varphi \leq \frac{7\pi}{4}\right\rbrace. \end{split}
\end{align}
 Note that $\mathcal{A}_{+}$ and $\mathcal{A}_{-}$ are symmetrical with respect to the plane $z=0$ and 
 \begin{align}(\mathcal{A}_{+}\vee \mathcal{A}_{-})\cap \vee_{k=1}^{\infty} \mathcal{M}_{2k-1}=\lbrace 0\rbrace.\end{align} Since regions $\bar{\mathcal{A}}_{\pm}\subset \mathcal{A}_{\pm}$ are homeomorphic to $\mathbb{R}^{3}_{+}$, we define two maps $f_{\pm}\colon \bar{\mathcal{A}}_{\pm}\to \bar{\mathcal{A}}_{\pm}$ in the same way as in Section \ref{sec3}. 
 \begin{figure}[!h] 
\def\svgwidth{0.65\columnwidth}
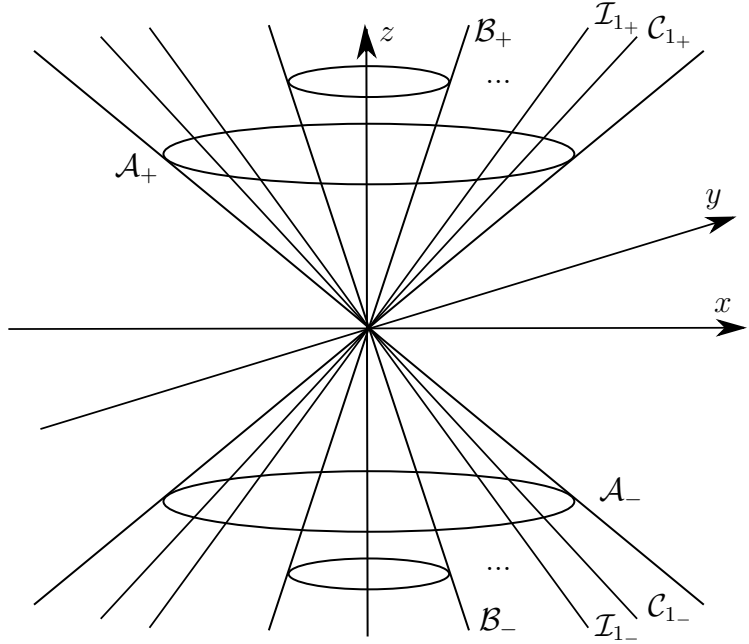
\caption{The solid cones $\mathcal{A}_{+}$ and $\mathcal{A}_{-}$ with an infinite number of the main $\mathcal{C}_{k_{\pm}}$ and the intermediate $\mathcal{I}_{k_{\pm}}$ surfaces inside.} 
\label{uv}
\end{figure}
 The construction of such maps is obvious, so we give only the values of integers $p_{k}$, which specify the indices of maps. We also assume that the integers $p_{k}$ are the same for each of the maps $f_{+}$ and $f_{-}$. Let $p_{2(2k-1)}=1$ for $k>0$, $p_{2k}=1+ka_{2k}$ for even $k>0$ and \begin{align}\label{int}
    p_{2k-1}=\begin{cases}
    1+a_{2} \ &\hbox{if $a_{1}<0$ and $k=1$},\\
    1+(a_{2}+a_{1}-1) \ &\hbox{if $a_{1}>0$ and $k=1$},\\
        1+a_{2(2k-1)}(2k-1) \ &\hbox{if $a_{2k-1}<0$ and $k> 1$},\\
        1+(a_{2(2k-1)}+a_{2k-1})(2k-1) &\hbox{if $a_{2k-1}>0$ and $k>1$.}
    \end{cases}
\end{align}  
Now, let $(\theta_{1})_{\pm}$ denote latitudes of the surfaces $(C_{1})_{\pm}\subset \mathcal{A}_{\pm}$. We define the map $\Phi\colon \mathcal{A}_{+}\vee \mathcal{A}_{-}\to \mathcal{A}_{+}\vee \mathcal{A}_{-}$ (in the same conical coordinates as given in Subsection \ref{constr}) by the following system of equations:
 \begin{align}\label{mappm}
    \Phi(x, \theta)=\begin{cases}
        f_{+}(x,\theta) \ &\hbox{for $\theta\in [(\theta_{1})_{+},\frac{\pi}{2}$]}, \\
         (x,\bar{\mu}(\theta))\circ S(x,\omega(\theta))\circ H_{p_{1}}(x,1-\omega(\theta))&\hbox{for $\theta \in [\frac{\pi}{4},(\theta_{1})_{+}]$},\\
         f_{-}(x, \theta) \ &\hbox{for $\theta\in [-\frac{\pi}{2},(\theta_{1})_{-}$]},\\
         (x,\bar{\mu}(\theta))\circ S(x,\omega(\theta)) \circ H_{p_{1}}(x,1-\omega(\theta)) &\hbox{for $\theta \in [(\theta_{1})_{-},
        -\frac{\pi}{4}]$},
    \end{cases}
\end{align}
where $\omega\colon [(\theta_{1})_{-},
        -\frac{\pi}{4}]\cup [\frac{\pi}{4},(\theta_{1})_{+}]\to [0,1]$ and $\bar{\mu}\colon[(\theta_{1})_{-},
        -\frac{\pi}{4}]\cup [\frac{\pi}{4},(\theta_{1})_{+}]\to [(\theta_{1})_{-},
        -\frac{\pi}{4}]\cup [\frac{\pi}{4},(\theta_{1})_{+}]$ are maps whose graphs are pictured in Figure \ref{twonewmaps}.
\begin{figure}[h!]
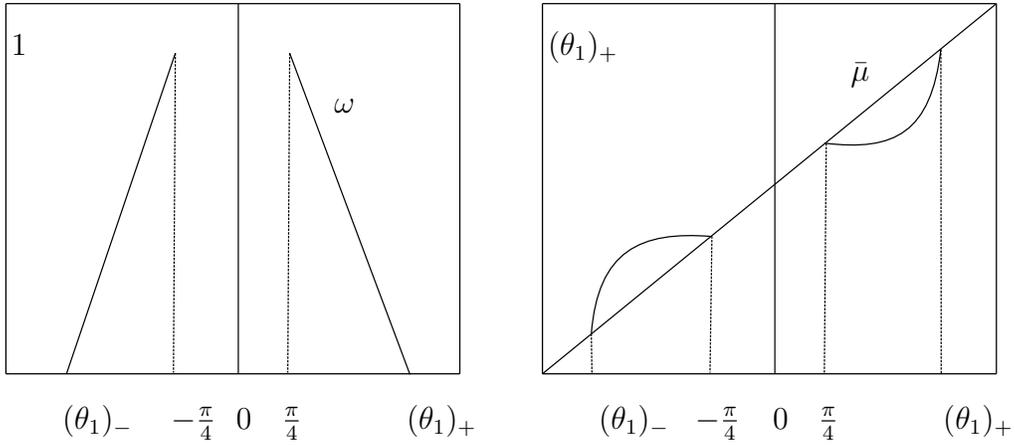
 
\def\svgwidth{0.45\columnwidth} 
\begingroup%
  \makeatletter%
  \providecommand\color[2][]{%
    \errmessage{(Inkscape) Color is used for the text in Inkscape, but the package 'color.sty' is not loaded}%
    \renewcommand\color[2][]{}%
  }%
  \providecommand\transparent[1]{%
    \errmessage{(Inkscape) Transparency is used (non-zero) for the text in Inkscape, but the package 'transparent.sty' is not loaded}%
    \renewcommand\transparent[1]{}%
  }%
  \providecommand\rotatebox[2]{#2}%
  \newcommand*\fsize{\dimexpr\f@size pt\relax}%
  \newcommand*\lineheight[1]{\fontsize{\fsize}{#1\fsize}\selectfont}%
  \ifx\svgwidth\undefined%
    \setlength{\unitlength}{317.48031496bp}%
    \ifx\svgscale\undefined%
      \relax%
    \else%
      \setlength{\unitlength}{\unitlength * \real{\svgscale}}%
    \fi%
  \else%
    \setlength{\unitlength}{\svgwidth}%
  \fi%
  \global\let\svgwidth\undefined%
  \global\let\svgscale\undefined%
  \makeatother%
  \begin{picture}(1,0.84821429)%
    \lineheight{1}%
    \setlength\tabcolsep{0pt}%
    \put(0,0){\includegraphics[width=\unitlength,page=1]{omegamap.pdf}}%
    \put(0.47254793,0.00994469){\makebox(0,0)[lt]{\lineheight{1.25}\smash{\begin{tabular}[t]{l}$0$\end{tabular}}}}%
    \put(0.13527421,0.00994469){\makebox(0,0)[lt]{\lineheight{1.25}\smash{\begin{tabular}[t]{l}$(\theta_{1})_{-}$\end{tabular}}}}%
    \put(0.34633992,0.00994469){\makebox(0,0)[lt]{\lineheight{1.25}\smash{\begin{tabular}[t]{l}$-\frac{\pi}{4}$\end{tabular}}}}%
    \put(0.56074196,0.00994469){\makebox(0,0)[lt]{\lineheight{1.25}\smash{\begin{tabular}[t]{l}$\frac{\pi}{4}$\end{tabular}}}}%
    \put(0.80261785,0.00994469){\makebox(0,0)[lt]{\lineheight{1.25}\smash{\begin{tabular}[t]{l}$(\theta_{1})_{+}$\end{tabular}}}}%
    \put(0.03640727,0.73558075){\makebox(0,0)[lt]{\lineheight{1.25}\smash{\begin{tabular}[t]{l}$1$\end{tabular}}}}%
    \put(0.66267742,0.62254844){\makebox(0,0)[lt]{\lineheight{1.25}\smash{\begin{tabular}[t]{l}$\omega$\end{tabular}}}}%
    \put(0,0){\includegraphics[width=\unitlength,page=2]{omegamap.pdf}}%
  \end{picture}%
\endgroup%

    \centering
    \subfloat{{\def\svgwidth{0.45\columnwidth}
\begingroup%
  \makeatletter%
  \providecommand\color[2][]{%
    \errmessage{(Inkscape) Color is used for the text in Inkscape, but the package 'color.sty' is not loaded}%
    \renewcommand\color[2][]{}%
  }%
  \providecommand\transparent[1]{%
    \errmessage{(Inkscape) Transparency is used (non-zero) for the text in Inkscape, but the package 'transparent.sty' is not loaded}%
    \renewcommand\transparent[1]{}%
  }%
  \providecommand\rotatebox[2]{#2}%
  \newcommand*\fsize{\dimexpr\f@size pt\relax}%
  \newcommand*\lineheight[1]{\fontsize{\fsize}{#1\fsize}\selectfont}%
  \ifx\svgwidth\undefined%
    \setlength{\unitlength}{317.48031496bp}%
    \ifx\svgscale\undefined%
      \relax%
    \else%
      \setlength{\unitlength}{\unitlength * \real{\svgscale}}%
    \fi%
  \else%
    \setlength{\unitlength}{\svgwidth}%
  \fi%
  \global\let\svgwidth\undefined%
  \global\let\svgscale\undefined%
  \makeatother%
  \begin{picture}(1,0.84821429)%
    \lineheight{1}%
    \setlength\tabcolsep{0pt}%
    \put(0,0){\includegraphics[width=\unitlength,page=1]{mu1map.pdf}}%
    \put(0.47254793,0.00994469){\makebox(0,0)[lt]{\lineheight{1.25}\smash{\begin{tabular}[t]{l}$0$\end{tabular}}}}%
    \put(0.13527421,0.00994469){\makebox(0,0)[lt]{\lineheight{1.25}\smash{\begin{tabular}[t]{l}$(\theta_{1})_{-}$\end{tabular}}}}%
    \put(0.32198804,0.00994469){\makebox(0,0)[lt]{\lineheight{1.25}\smash{\begin{tabular}[t]{l}$-\frac{\pi}{4}$\end{tabular}}}}%
    \put(0.56074196,0.00994469){\makebox(0,0)[lt]{\lineheight{1.25}\smash{\begin{tabular}[t]{l}$\frac{\pi}{4}$\end{tabular}}}}%
    \put(0.80261785,0.00994469){\makebox(0,0)[lt]{\lineheight{1.25}\smash{\begin{tabular}[t]{l}$(\theta_{1})_{+}$\end{tabular}}}}%
    \put(0.03640727,0.73558075){\makebox(0,0)[lt]{\lineheight{1.25}\smash{\begin{tabular}[t]{l}$(\theta_{1})_{+}$\end{tabular}}}}%
    \put(0.62371438,0.68464577){\makebox(0,0)[lt]{\lineheight{1.25}\smash{\begin{tabular}[t]{l}$\bar{\mu}$\end{tabular}}}}%
    \put(0,0){\includegraphics[width=\unitlength,page=2]{mu1map.pdf}}%
  \end{picture}%
\endgroup%
}}%
\caption{The graphs of the maps $\omega$ and $\bar{\mu}$.}
\label{twonewmaps}
\end{figure}
\begin{remark}
    The map $\Phi$ given in (\ref{mappm}) is a continuous extension of the map $f_{+}\vee f_{-}$ to the self-map of $\mathcal{A}_{+}\vee \mathcal{A}_{-}$ that has defined the sink dynamics of $S(p,1)$ on the boundary without changing the fixed point index of a map.
\end{remark}
\subsection{Construction of the map}
 We define the map $g\colon \mathbb{R}^{3}\to \mathbb{R}^{3}$ by the following equality: \begin{align}\label{defmap}
    g(p)=\begin{cases}
        \Theta(p) \ &\hbox{for $p\in \vee_{k=1}^{\infty}\mathcal{M}_{2k-1}$}, \\
        \Phi(p) \ &\hbox{for $p\in \mathcal{A}_{+}\vee \mathcal{A}_{-}$},   \\
        S(p,1) \ &\hbox{otherwise}.
    \end{cases}
\end{align}
Note that the map $g$ is a well-defined orientation-preserving homeomorphism of $\mathbb{R}^{3}$. Finally, we define the desire map, an orientation-reversing homeomorphism of $\mathbb{R}^{3}$, as follows. \begin{definition}\label{map}
    Let $p=(x,y,z)$ and $s(x,y,z)=(x,y,-z)$ denotes the reflection map with respect to the plane $z=0$. We define an orientation-reversing homeomorphism of $\mathbb{R}^{3}$ as the following composition $f\colon \mathbb{R}^{3}\to \mathbb{R}^{3}$, $f=s\circ g$.
\end{definition} 
We will show in Section \ref{proof} that $f$ realizes (\ref{53}).
\subsection{Computation of $\ind(s\circ \breve{g}_{h}\circ R,0)$ and $\ind(s\circ \breve{g}_{e}\circ R,0)$} 
For the further calculations, we need to calculate $\ind(s\circ \breve{g}_{h}\circ R,0)$ and $\ind(s\circ \breve{g}_{e}\circ R,0)$, where $R$ denotes a retraction of $\mathbb{R}^{3}$ onto $\breve{\mathcal{S}}$.

First, let us state the following lemma. \begin{lemma}\label{sym}
    Let $s(x,y,z)=(x,y,-z)$ and $g\colon \mathbb{R}^{3}\to \mathbb{R}^{3}$ be an orientation-preserving homeomorphism of $\mathbb{R}^{3}$ such that $g(\mathbb{R}^{3}_{-})\subseteq (\mathbb{R}^{3}_{-})$  and $g(\mathbb{R}^{3}_{+})\subseteq (\mathbb{R}^{3}_{+})$. Then \begin{align}\ind((s\circ g)^{n},0)=\ind((\restr{g}{z=0})^{n},0)\end{align} for each odd $n$. 
\end{lemma}
\begin{proof}
  Let $U$ be an open neighborhood of $0$ and $H_{t}(x,y,z)\colon U\times [0,1]\to \mathbb{R}^{3}$ a standard linear homotopy between $(s\circ g)^{n}$ and $g^{n}\circ R$, where $R(x,y,z)=(x,y,0)$. Note that: 
  \begin{align}
  \begin{split}
      & \hbox{if} \ (x,y,z)\in \mathbb{R}^{3}_{+}\cap U\implies H^{(n)}_{t}(x,y,z)\in \mathbb{R}^{3}_{-}  \ \hbox{for} \ t\in [0,1],\\
      &\hbox{if} \ (x,y,z)\in \mathbb{R}^{3}_{-}\cap U\implies H^{(n)}_{t}(x,y,z)\in \mathbb{R}^{3}_{+}  \ \hbox{for} \ t\in [0,1].
      \end{split}
  \end{align}
 Thus, the point $0$ is the only fixed point of $H_{t}$ in $U$, hence we immediately obtain: \begin{align}
      \ind((s\circ g)^{n},0)=\ind(g^{n}\circ R,0)=\ind((\restr{g}{z=0})^{n},0).
  \end{align}
\end{proof}
\begin{lemma}\label{index_0}
$\ind(s\circ \breve{g}_{h}\circ R,0)=0$ and $\ind(s\circ \breve{g}_{e}\circ R,0)=2$.
\end{lemma}

\begin{proof}
Applying Lemma \ref{sym} and Poincaré-Bendixson formula (see Figure \ref{dynamicsplane}) we get thesis. The same result is satisfied for any odd iteration of such maps. 
\end{proof}

\begin{figure}[!h] 
\def\svgwidth{1\columnwidth}
\begingroup%
  \makeatletter%
  \providecommand\color[2][]{%
    \errmessage{(Inkscape) Color is used for the text in Inkscape, but the package 'color.sty' is not loaded}%
    \renewcommand\color[2][]{}%
  }%
  \providecommand\transparent[1]{%
    \errmessage{(Inkscape) Transparency is used (non-zero) for the text in Inkscape, but the package 'transparent.sty' is not loaded}%
    \renewcommand\transparent[1]{}%
  }%
  \providecommand\rotatebox[2]{#2}%
  \newcommand*\fsize{\dimexpr\f@size pt\relax}%
  \newcommand*\lineheight[1]{\fontsize{\fsize}{#1\fsize}\selectfont}%
  \ifx\svgwidth\undefined%
    \setlength{\unitlength}{352.5bp}%
    \ifx\svgscale\undefined%
      \relax%
    \else%
      \setlength{\unitlength}{\unitlength * \real{\svgscale}}%
    \fi%
  \else%
    \setlength{\unitlength}{\svgwidth}%
  \fi%
  \global\let\svgwidth\undefined%
  \global\let\svgscale\undefined%
  \makeatother%
  \begin{picture}(1,0.38297872)%
    \lineheight{1}%
    \setlength\tabcolsep{0pt}%
    \put(0,0){\includegraphics[width=\unitlength,page=1]{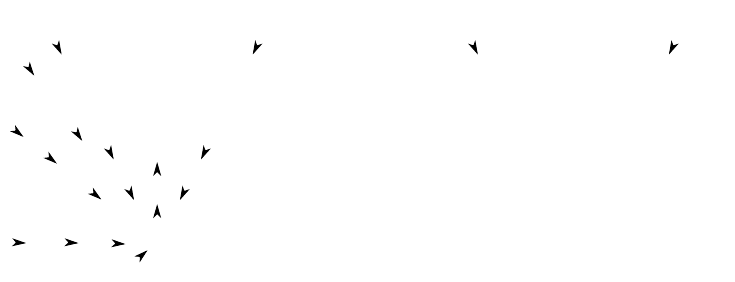}}%
    \put(0.20815629,0.00544391){\makebox(0,0)[lt]{\lineheight{1.25}\smash{\begin{tabular}[t]{l}$0$\end{tabular}}}}%
    \put(0.77490637,0.00544304){\makebox(0,0)[lt]{\lineheight{1.25}\smash{\begin{tabular}[t]{l}$0$\end{tabular}}}}%
    \put(0,0){\includegraphics[width=\unitlength,page=2]{FiniteConesplane.pdf}}%
    \put(0.23120958,0.34068229){\makebox(0,0)[lt]{\lineheight{1.25}\smash{\begin{tabular}[t]{l}$r=r_{0}$\end{tabular}}}}%
    \put(0.79407543,0.34068253){\makebox(0,0)[lt]{\lineheight{1.25}\smash{\begin{tabular}[t]{l}$r=r_{0}$\end{tabular}}}}%
    \put(0,0){\includegraphics[width=\unitlength,page=3]{FiniteConesplane.pdf}}%
  \end{picture}%
\endgroup%

\caption{The hyperbolic and elliptic dynamics of $\restr{(\breve{g}_{h}\circ R)}{z=0}$ and $\restr{(\breve{g}_{e}\circ R)}{z=0}$, respectively.}
\label{dynamicsplane}
\end{figure}

\section{Proof of Theorem \ref{thm3}}\label{proof}
We show that the map $f$ given in Definition \ref{map} (see also formula \ref{defmap}) provides the realization of an arbitrary sequence of integers.
First, note that by the Mayer-Vietoris formula (see Theorem \ref{May-Viet}) and definition of the fixed point index, the equation: \begin{align}\label{form}
    \ind(f^{n},0)=\ind((s\circ \Theta)^{n},0)+\ind((s\circ \Phi)^{n},0)-1
\end{align}
holds for each $n$. 

We start with calculation of the indices of iterates of $s\circ \Theta$.

Let us point a few important facts about the constructed map that follows directly from the definition of the fixed point index: \begin{itemize} \item if $n$ is not divided by $2k-1$, then  (see formula (\ref{Theta})) $(s\circ\Theta_{2k-1})^{n}$ is homotopic to a map that generates a sink dynamics. 
\item $\ind((s\circ\Theta_{2k-1})^{n},0)=\ind(\restr{(s\circ \Theta_{2k-1})^{n}}{\vee_{i=1}^{2k-1}\breve{\mathcal{S}}_{i}},0)$. 
\item $\ind(\restr{(s\circ \Theta)^{n}}{\mathcal{M}_{1}},0)=\ind(\restr{(s\circ \Theta)^{n}}{\vee_{j=1}^{|a_{1}|\pm 1}(\mathcal{J}_{1})_{j}},0)$.
\item $\ind(\restr{(s\circ \Theta)^{n}}{\mathcal{M}_{2k-1}},0)=\ind(\restr{(s\circ \Theta)^{n}}{\vee_{j=1}^{|a_{2k-1}|}(\mathcal{J}_{2k-1})_{j}},0)$ for $k>1$.
\end{itemize}

The calculation of indices of iterates of $\ind((s\circ\Theta_{2k-1})^{n},0)$ reduces to determining it for $2k-1$ cones $\mathcal{\breve{S}}_{i}$ (cf. (\ref{newflow})) which are symmetrical with respect to $z=0$. Moreover, for each cone and odd iterate, we have: $\ind(\restr{(s\circ \Theta_{2k-1})^{n}}{\breve{\mathcal{S}}_{i}},0)$ is equal either $0$ or $2$ (see Lemma \ref{index_0}).

We consider the following cases:
\begin{enumerate}
\item[{\bf (I)}] \underline{If $a_{2k-1}<0$ and $k>1$}
\end{enumerate}

Using the Mayer-Vietoris formula (\ref{May-viet1}) inductively for consecutive solid cones in a bouquet, Lemma \ref{sym} and the above facts, we calculate that: \begin{align}
        \ind((s\circ \Theta_{2k-1})^{n},0)=\begin{cases}
        1-(2k-1) \ &\hbox{if $(2k-1)|n$},\\
        1 \ &\hbox{otherwise}.\end{cases}
\end{align}
Hence the equation $\ind((s\circ \Theta_{2k-1})^{n},0)=\reg_{1}(n)-\reg_{2k-1}(n)$ is satisfied. Let us now extend this result to the map $\restr{(s\circ \Theta)}{\mathcal{M}_{2k-1}}\colon \mathcal{M}_{2k-1}\to \mathcal{M}_{2k-1}$. The analogous calculation gives us the following: \begin{align}
    \ind(\restr{(s\circ \Theta)^{n}}{\mathcal{M}_{2k-1}},0)=\reg_{1}(n)+a_{2k-1}\reg_{2k-1}(n).
\end{align}
\begin{enumerate}
\item[{\bf (II)}] \underline{If $a_{1}<0$}
\end{enumerate}
In this case, directly by the definition of the map $\restr{\Theta}{\mathcal{M}_{1}}\colon \mathcal{M}_{1}\to \mathcal{M}_{1}$, we obtain: \begin{align}
    \ind(\restr{(s\circ \Theta)^{n}}{\mathcal{M}_{1}},0)=a_{1}\reg_{1}(n).
\end{align}
\begin{enumerate}
\item[{\bf (III)}] \underline{If $a_{2k-1}>0$ and $k>1$}
\end{enumerate}
As in the first case \textbf{(I)}, we start with the map $(s\circ \Theta_{2k-1})\colon \mathcal{J}_{2k-1}\to \mathcal{J}_{2k-1}$. According to the previous arguments, we get: \begin{align}
    \ind((s\circ \Theta_{2k-1})^{n},0)=\begin{cases}
        1+(2k-1) \ &\hbox{if $(2k-1)|n$ and $n$ is odd},\\
        1-(2k-1) \ &\hbox{if $(2k-1)|n$ and $n$ is even},\\
        1 \ &\hbox{otherwise}.
    \end{cases}
\end{align}
Thus $\ind((s\circ \Theta_{2k-1})^{n},0)=\reg_{1}(n)+\reg_{2k-1}(n)-\reg_{2(2k-1)}(n)$ is satisfied. Now, as before, we extend this result and get the following: \begin{align}
    \ind(\restr{(s\circ \Theta)^{n}}{\mathcal{M}_{2k-1}},0)=\reg_{1}(n)+a_{2k-1}\reg_{2k-1}(n)-a_{2k-1}\reg_{2(2k-1)}(n).
\end{align}
\begin{enumerate}
\item[{\bf (IV)}] \underline{If $a_{1}>0$}
\end{enumerate}
Similarly to \textbf{(II)} by the definition of the map  $\restr{\Theta}{\mathcal{M}_{1}}\colon \mathcal{M}_{1}\to \mathcal{M}_{1}$, we immediately
obtain: \begin{align}
    \ind(\restr{(s\circ \Theta)^{n}}{\mathcal{M}_{1}},0)=a_{1}\reg_{1}(n)+(1-a_{1})\reg_{2}(n).
\end{align}
Now define the map $\rho\colon \lbrace a_{1}\rbrace \to \lbrace 0,1\rbrace$ dependent on the value of $a_{1}$ as follows: $\rho(a_{1})=1$ if $a_{1}>0$ and $\rho(a_{1})=0$ otherwise. Consequently, adding up all cases, we get the following:
\begin{align}
\begin{split}
    \ind((s\circ \Theta)^{n},0)&=a_{1}\reg_{1}(n)+\rho(a_{1})(1-a_{1})\reg_{2}(n) +\sum_{a_{2k-1}<0, k>1} a_{2k-1}\reg_{2k-1}(n)\\&+\sum_{a_{2k-1}>0, k>1} (a_{2k-1}\reg_{2k-1}(n)-a_{2k-1}\reg_{2(2k-1)}(n)).
    \end{split}
\end{align}

Now we make a calculation of the indices of iterates of  $s\circ \Phi$.
First, let us make an obvious observation. If $n$ is an odd integer, then $(s\circ \Phi)(\mathcal{A}_{\pm})\subset\mathcal{A}_{\mp}$ and thus $\ind((s\circ \Phi)^{n},0)=1$ for such iterates. It remains to perform calculations for even iterates due to integers $p_{k}$ (see \ref{int}). Once again, using the Mayer-Vietoris formula, we get the following: \begin{align}\begin{split}
    \ind((s\circ \Phi)^{n},0)&=\reg_{1}(n)+\reg_{2}(n)(a_{2}+\rho(a_{1})(a_{1}-1))+\sum_{a_{2k-1}<0,k>1}a_{2(2k-1)}\reg_{2(2k-1)}(n)\\
    &+\sum_{a_{2k-1}>0,k>1}(a_{2(2k-1)}+a_{2k-1})\reg_{2(2k-1)}(n)+\sum_{k>1, 2|k}a_{2k}\reg_{2k}(n).
    \end{split}
\end{align}

Finally, with reference to (\ref{form}), we obtain: \begin{align}
    \ind(f^{n},0)=\sum_{k=1}^{\infty}a_{k}\reg_{k}(n).
\end{align}
This completes the proof of Theorem \ref{thm3} in $\mathbb{R}^{3}$. \vskip 5mm 

In  dimension higher than $3$ there are also no restrictions for indices. 
Namely, we can now consider $f^{*}\colon \mathbb{R}^{m-3}\to \mathbb{R}^{m-3}$ for $m>3$, an orientation-preserving homeomorphism such that $\lbrace 0\rbrace$ is a global attractor. Then $\ind((f^{*})^{n},0)=\reg_{1}(n)$. The product map $(f\times f^{*})\colon \mathbb{R}^{m}\to \mathbb{R}^{m}$ is an orientation-reversing self-homeomorphism of $\mathbb{R}^{m}$, using the product formula for the fixed point index, we get the main result: \begin{align}
    \ind((f\times f^{*})^{n},0)=\sum_{k=1}^{\infty}a_{k}\reg_{k}(n),
\end{align}
which proves Theorem \ref{thm3} in $m-$dimensional space.
\section{Final remarks}
Observe that in our construction of the map realizing indices the fixed point $0$ is a limit of periodic points of $f$ (i.e. $0$ is an accumulated fixed point of $f$, see Definition \ref{accumulated}). 
In case $\lbrace 0\rbrace$ is a non-accumulated fixed point of $f$ we could provide directly from our construction the realization only for periodic expansion with finitely many non-zero odd terms.  

\begin{corollary} \label{corac}
    For any sequence $\lbrace a_{n}\rbrace_{n=1}^{\infty}$ of Dold's coefficients, there is an orientation-reversing homeomorphism $f$ of $\mathbb{R}^{3}$ such that $\lbrace 0\rbrace$ is a non-accumulated fixed point and \begin{align} \label{eq527}
    \ind(f^{n},0)=\sum_{k=1}^{\infty}a_{k}\reg_{k}(n)
\end{align}
if  there are finitely many non-zero $a_{k}$ for odd $k$. 
\end{corollary}
This leads us to formulate the following question.
\begin{problem}
    Is the Corollary \ref{corac} optimal? In other words, under the assumption that $\lbrace 0\rbrace $ is a non-accumulated fixed point of  an orientation-reversing homeomorphism, is it possible to realize any sequence of Dold's coefficients in (\ref{eq527}). 
\end{problem}

\bibliographystyle{elsarticle-num}

\end{document}